\documentclass[preprint,12pt]{elsarticle}
\usepackage{latexsym}
\usepackage{amssymb}
\usepackage{amsmath}
\usepackage{amsthm}
\usepackage{verbatim}
\usepackage{appendix}
\usepackage{epsfig}
\usepackage[text={145mm,220mm},centering]{geometry}
\usepackage{color}
\usepackage{hyperref}
\numberwithin{equation}{section}

\newtheorem{thm}{Theorem}[section]
\newtheorem{lem}{Lemma}[section]

\newtheorem{rem}{Remark}[section]
\newtheorem{defi}{Definition}[section]
\newtheorem{corollary}{Corollary}[section]

\newcommand{\RR}{\mathbb{R}}

\newcommand{\NN}{\mathcal{N}}

\journal{}

\begin{document}
\UseRawInputEncoding
\allowdisplaybreaks[4]
\begin{frontmatter}

\title{Blow-up and decay for  a class of variable coefficient wave equation with nonlinear damping and logarithmic source
}

\author[ad1]{Pengxue Cui}
\ead{cuipx0205@vip.163.com}
\author[ad1]{Shuguan Ji\corref{cor}}
\ead{jisg100@nenu.edu.cn}
\address[ad1]{School of Mathematics and Statistics and Center for Mathematics and Interdisciplinary Sciences, Northeast Normal University, Changchun 130024, P.R. China}
\cortext[cor]{Corresponding author.}

\begin{abstract}
In this paper, we consider the long time behavior for the solution of a class of variable coefficient wave equation with nonlinear damping and logarithmic source.
The existence and uniqueness of local weak solution can be obtained by using the Galerkin method and contraction mapping principle.
However, the long time behavior of the solution is usually complicated and it depends on the balance mechanism between the damping and source terms. When the damping exponent $(p+1)$ (see assumption (H3)) is greater than the source term exponent $(q-1)$ (see equation (\ref{1.1})), namely, $p+2>q$, we obtain the global existence and accurate decay rates of the energy for the weak solutions with any initial data.
Moreover, whether the weak solution exists globally or blows up in finite time, it is closely related to the initial data. In the framework of modified potential well theory, we construct the stable and unstable sets (see \eqref{e2.12}) for the initial data. For the initial data belonging to the stable set, we prove that the weak solution exists globally and has similar decay rates as the previous results. For $p+2<q$ and the initial data belonging to the unstable set, we prove that the weak solution blows up in finite time for a little special damping $g(u_{t})=|u_{t}|^{p}u_{t}$.
\end{abstract}


\end{frontmatter}



\section{Introduction}

In this paper, we consider the following model for  a class of variable coefficient wave equation with nonlinear damping and logarithmic source
\begin{equation}\label{1.1}
\left\{\begin{array}{llll}
&u_{tt}-\mu(t) Lu+g(u_{t})=|u|^{q-2}u\log|u|,&x \in \Omega,&t>0,\\
&u(x,t)=0,  &x \in \partial\Omega,&t>0,\\
&u(x,0)=u_{0}(x), u_{t}(x,0)=u_{1}(x),  &x \in \Omega,&
\end{array}\right.
\end{equation}
where $\Omega\subset \RR^{n}$ $(n\geq1)$ is a bounded domain with smooth boundary $\partial\Omega$, $\mu(t)$ is a given function and $u_{0}\in H_{0}^{1}(\Omega)$, $u_{1}\in L^{2}(\Omega)$ are given initial data. The second order differential operator $L$ is given by
\begin{equation*}
  Lu =div(A(x)\nabla u)=\sum\limits_{i,j=1}^{\infty}\frac{\partial}{\partial x_{i}}(a_{ij}(x)\frac{\partial u}{\partial x_{j}}).
\end{equation*}
$q$ is a constant satisfying
\begin{equation}\label{1.2}
  2<q<2^{*}=
  \begin{cases}
  \frac{2n}{n-2}&\text{if $n\geq 3$},\\
  +\infty &\text{if $n=1,2$}.
  \end{cases}
\end{equation}

The wave equations with damping and source term are widely used in  fluid dynamics, classical mechanics and quantum field theory \cite{beijing1,beijing2}. Due to the complexity of the balance mechanism between damping and source term, it is not easy to study the long time behavior of the solutions. It is worth to mention that when the equation doesn't contain damping, the source term can make the solutions blow up in finite or infinite time (see Glassey \cite{nodamp1}, Levine \cite{nodamp2,Levine1} and references therein). When the equation doesn't contain source term, the damping can make the solutions exist globally (see Agre \cite{nosource1}, \'Ang \cite{nosource2} and others). Therefore, from a mathematical point of view, the research of the interaction between them becomes more difficult and attractive.

Among the extensive study of wave equations with polynomial source, let us go back to the work of Georgiev and Todorova \cite{JDE1994}. They considered the following wave equation
\begin{equation}\label{1.3}
  u_{tt}-\Delta u+ |u_{t}|^{p}u_{t}=|u|^{q}u, \ \ p,q>0,
\end{equation}
with the same initial and boundary values as in \eqref{1.1}. They obtained that if $q\leq p$ the solutions exist globally; if $q>p$, the solutions blow up in finite time. For more information about the variation of wave equations with polynomial source, see \cite{AR2006,GS2006,Ha15}  and references therein.

In the case of logarithmic source $|u|^{p-2}u\log|u|^{2}(p\geq 2)$, such wave equations are derived from the fields of quantum mechanics, nuclear physics, inflation cosmology,  super-symmetric field theory \cite{JP,KE,La}. For instance, in optical theory \cite{Bu03,Kr00}, when a partially coherent beam propagates in a nonlinear medium with logarithmic nonlinearity, the logarithmic source represents the refractive change index of light.
In quantum mechanics \cite{log1,log2}, since forces which cause dispersion are exactly balanced for our soliton by the negative pressure due to the logarithmic nonlinearity, the source can be introduced to describe the nonrelativistic wave equation of spin particles moving in an external electromagnetic field (Pauli equation) and the relativistic wave equation of spinless particles (Klein-Gordon equation). The Klein-Gordon equation with logarithmic source was introduced by Birula and Mycielski \cite{log1,log2},
\begin{equation}\label{1.4}
 u_{tt}- u_{xx}+u-\varepsilon u\log |u|^{2}=0,
\end{equation}
with the initial and boundary values as in \eqref{1.1}. Cazenave and Haraux \cite{CH} studied the existence and uniqueness of the local solution of (\ref{1.4}) in $\mathbb{R}^{3}$ by using Galerkin method. Subsequently, G\'orka \cite{dalog5} established the global existence of weak solutions of (\ref{1.4}) in $\mathbb{R}$. Recently, Hu et al. \cite{hu2019} studied the following equation
\begin{equation*}
  u_{tt}-\Delta u+u+u_{t}=u \log|u|^{q}, \ \ 0<q<1,
\end{equation*}
with the initial and boundary values as in \eqref{1.1}. By using potential well method, they estimated the decay rate of the energy and proved that the solutions blow up at $+\infty$ in the $H_{0}^{1}(\Omega)$ norm.  Chen and Tian \cite{dalog7} considered a semi-linear pseudo-parabolic equation with logarithmic source $u \log|u|$. By comparing with \cite{XUJFA}, they found that polynomial source $|u|^{p-1}u$ plays a significant role in the result of blow-up in finite time for the solutions of such semi-linear pseudo-parabolic equation. For more information of the equations with logarithmic source, see \cite{Al19,Al20,Ch20,2018ma,dalog6} and the references therein. Therefore, it is meaningful to consider logarithmic source $|u|^{q-2}u \log|u|$

For the case of variable coefficient, such wave equations arise from the propagation of seismic waves in non-isotropic media, the flow of fluid in porous media \cite{Barbu1,Barbu2}, and have been studied in many literatures (see \cite{HADCDS,HATAI2017,HAO2019} and the references therein). However, few literatures analysed the long time behavior for the weak solution of  the variable coefficient wave equations when it is related to the interaction between damping and source term.
Recently, Ha \cite{HAMO2021} studied the wave equation
\begin{equation}\label{e1.3}
u_{tt}-\mu(t) Lu+g(u_{t})=|u|^{q}u,
\end{equation}
with the same initial and boundary values as in \eqref{1.1}.  They found that the long time behavior of the solutions is not only affected by the balance mechanism between damping and source term, but also closely related to the initial data. When the source term is supercritical, the author established the sufficient conditions for the existence, decay and blow-up in finite time of the weak solutions, respectively. Hao and Du \cite{HAO2022} studied a variable-coefficient viscoelastic wave equation with logarithmic nonlinearity $u\log|u|$, linear damping $u_{t}$ and nonlinear boundary condition. They obtained the general decay rate of energy by using Riemannian geometry method and blow-up at $+\infty$ by concavity method. But, to my knowledge, there is none literature for the variable coefficient wave equations with nonlinear damping and logarithmic source $|u|^{q-2}u \log|u|$.

Motivated by the above studies, in this paper, we mainly analyse the relationship between global existence or blow-up in finite time of weak solutions and the balance mechanism between nonlinear damping and logarithmic source. It is worth pointing out that although the logarithmic source is weaker than the polynomial source in some way, the theory of establishing the results of existence and energy decay in the case of polynomial source cannot be directly used.
The outline of the paper is as follows. In Section 2, we shall give some notations, lemmas and main results. In Section 3, the existence and uniqueness of local weak solutions is obtained by utilizing the standard Galerkin method and contraction mapping principle. In Section 4, the global existence of the weak solutions and the accurate decay rates of the energy are established, which are shown in Theorems \ref{theorem3.2} and \ref{theorem4.2}. In Section 5, that the weak solutions blow up in finite time with nonpositive initial energy and positive initial energy is proved, which is shown in Theorems \ref{theorem5.1} and \ref{theorem5.2}, respectively.

\section{Some notations and main results}

In this section, we give some assumptions, lemmas and main results. Throughout this paper, the norm of $f \in L^{s}(\Omega)$ is defined by $\| f \|_{s}$ $=(\int_{\Omega} |f|^{s} \textrm{d}x)^{1/s}$  for $1\leq s <\infty$. $H^{1}(\Omega), H_{0}^{1}(\Omega)$ denote the usual Sobolev spaces with the norm
\begin{equation*}
  \|u\|_{H^{1}(\Omega)}=\left( \|u\|_{2}^{2}+\|\nabla u\|_{2}^{2}\right)^{1/2},
\end{equation*}
and $H^{-1}(\Omega)$ indicates the dual space of $H_{0}^{1}(\Omega)$.

Let us begin with the assumptions of $A(x), \mu(t)$ and $g(s)$.

({\bf{H1}}) $A(x)=(a_{ij}(x)), x\in \Omega$, denotes symmetric matrices with $a_{ij}(x)\in C^{1}(\bar{\Omega})$, and there exists a positive constant $a_{0}$ such that
\begin{equation}\label{2.1}
  \sum_{i,j=1}^{n}a_{ij}(x)w_{i}w_{j}\geq a_{0}|w|^{2}, \ \ x\in \bar{\Omega}, \ \ 0\neq w=(w_{1},\ldots,w_{n}) \in \RR^{n}.
\end{equation}

({\bf{H2}}) Let $\mu(t)\in C^{1}(0, \infty)\cap W^{1,\infty}(0,\infty)$ satisfy
\begin{equation*}
  \mu(t)\geq \mu_{0}>0 \mbox{ and } \mu'(t)<0,~ \mbox{ a.e. } t\in[0,\infty),
\end{equation*}
where $\mu_{0}$ is a positive constant.

({\bf{H3}}) Let $g(s):\RR \rightarrow \RR$ be a nondecreasing $C^{1}$ function with $g(0)=0$ and assume that
\begin{eqnarray*}
  &|\beta(s)|\leq |g(s)|\leq |\beta^{-1}(s)|, &\mbox{ if } |s|\leq 1,\nonumber\\
  &c_{1}|s|^{p+1}\leq |g(s)| \leq c_{2}|s|^{p+1}, &\mbox{ if }|s|>1.
\end{eqnarray*}
where $c_{1}, c_{2}, p$ are positive constants, $\beta\in C^{1}[-1,1]$ is a strictly increasing function, $\beta^{-1}$ is the inverse function of $\beta$.

\begin{rem}\label{remark2.3}
According to $({\bf{H1}})$, it is easy to check that the function $a(\cdot,\cdot):H_{0}^{1}(\Omega)\times H_{0}^{1}(\Omega)\rightarrow \RR$, defined by
\begin{equation*}
  a(u(t),v(t))=\sum\limits_{i,j=1}^{n}a_{ij}(x)\frac{\partial u(t)}{\partial x_{i}}\frac{\partial v(t)}{\partial x_{j}}\textrm{d}x=\int_{\Omega}A\nabla u \cdot \nabla v\textrm{d}x,
\end{equation*}
is symmetric and continuous. Moreover, by replacing $w$ with $\nabla u$ in \eqref{2.1}, we have
\begin{equation}\label{2.5}
  a(u(t),u(t))\geq a_{0}\int_{\Omega}\sum_{i=1}^{n}\left| \frac{\partial u}{\partial x_{i}} \right|^{2}\textrm{d}x=a_{0}\|\nabla u\|_{2}^{2}.
\end{equation}
\end{rem}

\begin{rem}\label{remark2.2}
The assumption $({\bf{H3}})$ implies that $sg(s)>0$ for all $s\neq 0$.
\end{rem}

Next, we define the following functionals associated with problem \eqref{1.1} on $H_{0}^{1}(\Omega)$:
the Nehari functional
\begin{equation*}
  I(u)=\mu(t)a(u(t),u(t))-\int_{\Omega}|u|^{q}\log(|u|)\textrm{d}x,
\end{equation*}
the potential energy functional
\begin{eqnarray}\label{2.8}
  J(u)&=&\frac{1}{2}\mu(t)a(u(t),u(t))-\frac{1}{q}\int_{\Omega}|u|^{q}\log |u| \textrm{d}x+\frac{1}{q^{2}}\|u\|_{q}^{q}\nonumber\\
  &=&\frac{1}{q}I(u)+\frac{1}{q^{2}}\|u\|_{q}^{q}+\frac{q-2}{2q}\mu(t)a(u(t),u(t)),
\end{eqnarray}
and the total energy function
\begin{equation}\label{2.6}
  E(t):=E(u,u_{t})=\frac{1}{2}\|u_{t}\|_{2}^{2}+\frac{1}{2}\mu(t)a(u(t),u(t))
  -\frac{1}{q}\int_{\Omega}|u|^{q}\log |u| \textrm{d}x+\frac{1}{q^{2}}\|u\|_{q}^{q}.
\end{equation}
The mountain pass value of $J(u)$ is defined as
\begin{equation}\label{2.10}
  d:=\inf_{u}\{\sup_{\kappa>0}J(\kappa u): u\in H_{0}^{1}(\Omega)\backslash\{0\}\},
\end{equation}
and we also define the Nehari manifold
\begin{equation*}\label{2.11}
  \NN=\{u\in H_{0}^{1}(\Omega)\backslash\{0\}: I(u)=0\}.
\end{equation*}
So, the potential well depth $d$ that is the mountain pass value of $J(u)$ defined in \eqref{2.10} can be rewritten as
\begin{equation}\label{2.12}
  d=\inf_{u\in \NN} J(u).
\end{equation}
By Lemma \ref{lemma2.6}, we know that there exist a positive constant $M$ such that
\begin{equation}\label{eq2.12}
  d\geq M:=\left( \frac{q-2}{2q} \right)\mu_{0}r_{*}^{2},
\end{equation}
where $r_{*}$ is defined in Lemma \ref{lemma2.5}.

For convenience, we define the stable set $W$ and the unstable set $V$:
\begin{eqnarray}\label{e2.12}
W&=&\{(u,u_{t})\in H_{0}^{1}(\Omega)\backslash\{0\}\times L^{2}(\Omega): a(u,u)<r_{*}^{2}~ \text{and}~ 0<E(u,u_{t})<M\},\nonumber\\
V&=&\{(u,u_{t})\in H_{0}^{1}(\Omega)\backslash\{0\}\times L^{2}(\Omega): a(u,u)> r_{*}^{2}~ \text{and}~ 0<E(u,u_{t})<M \}\nonumber\\
&&\bigcup \{(u,u_{t})\in H_{0}^{1}(\Omega)\backslash\{0\}\times L^{2}(\Omega): E(u,u_{t})\leq 0 \}.
\end{eqnarray}


To state our main results, let us give the definition of weak solution to \eqref{1.1}.
\begin{defi}\label{definition2.1}
A function $u$ is called a weak solution of \eqref{1.1} on $\Omega\times [0,T]$, if $u\in C([0,T]; H_{0}^{1}(\Omega))\cap C^{1}([0,T]; L^{2}(\Omega))$, $u_{t}\in L^{p+2}([0,T],\Omega)\cap L^{2}([0,T]; H_{0}^{1}(\Omega))$, $(u(0,x), u_{t}(0,x))=(u_{0}, u_{1})\in H_{0}^{1}(\Omega) \times L^{2}(\Omega)$ and $u$ verifies the identify
\begin{eqnarray*}
 &&\int_{\Omega}u_{t} \phi \textrm{d}x -\int_{\Omega}u_{1}\phi \textrm{d}x +\int_{0}^{t}\mu(t)a(u,\phi)\textrm{d}t +\int_{0}^{t}\int_{\Omega}g(u_{t})\phi \textrm{d}x\textrm{d}t\\
 &=&\int_{0}^{t}\int_{\Omega}|u|^{q-2}u\log |u|\phi \textrm{d}x\textrm{d}t,
\end{eqnarray*}
for any $\phi\in H_{0}^{1}(\Omega)$ and for almost everywhere $t \in [0,T]$.
\end{defi}

Now, we establish the local existence of a weak solution to \eqref{1.1}.
\begin{thm} \label{theorem3.1}
Suppose that the assumptions $({\bf{H1}})-({\bf{H3}})$ hold. If the initial data $(u_{0},u_{1}) \in H_{0}^{1}(\Omega)\times L^{2}(\Omega)$ and $c_{1}>\frac{1}{p+2}$, then there exists $T>0$ and a unique local weak solution of \eqref{1.1} over $[0,T]$.
\end{thm}

The next theorems state that the weak solution provided by theorem \ref{theorem3.1} can be extended to the whole interval $[0,+\infty)$ and give the accurate energy decay rates when the damping has a general growth near zero.

\begin{thm} \label{theorem3.2}
In addition to the conditions of theorem \ref{theorem3.1}, we assume that $(u_{0},u_{1})\in W,$ or $ q < p + 2$,
then the unique weak solution $u$ furnished by theorem \ref{theorem3.1} is a global weak solution and the existence time $T$ can be taken arbitrarily large.
\end{thm}

\begin{thm}\label{theorem4.2}
Suppose that the conditions of theorem \ref{theorem3.2} hold and $$\beta(s)=|s|^{p+1}~(p\geq 0).$$
Then, we have the following energy decay rates:

\textbf{Case 1:} If $p=0$, there exists a positive constant  $k_{1}$ such that
\begin{eqnarray}\label{4.1}
  E(t)\leq E(0)e^{-k_{1}[t-1]^{+}}.
\end{eqnarray}

\textbf{Case 2:} If $p>0$, there exists positive constants $k_{2}$ such that
\begin{eqnarray}\label{2.2}
  &E(t)\leq \Big(E(0)^{-\frac{p}{2}}
  +k_{2}^{-\frac{p}{2}}\frac{p}{2}[t-1]^{+}\Big)^{-\frac{2}{p+2}}.
\end{eqnarray}
\end{thm}



In order to state the blow-up result, we need to impose additional assumptions on the damping:
\vskip 0.1in

$({\bf{H3'}})$  $g(s)=|s|^{p}s$ is a polynomial function.
\vskip 0.1in

\begin{thm}\label{theorem5.1}
In addition to $({\bf{H1}}), ({\bf{H2}})$and $({\bf{H3'}})$, we assume that $q>p+2$, $u_{0}\in H_{0}^{1}(\Omega), u_{1}\in L^{2}(\Omega)$ and $E(0)\leq 0$.
Then, the weak solution $u$ of \eqref{1.1}  furnished by  theorem \ref{theorem3.1} blows up in finite time.
\end{thm}

\begin{thm}\label{theorem5.2}
In addition to $({\bf{H1}}), ({\bf{H2}})$ and $({\bf{H3'}})$, we assume that $q>p+2$, $u_{0}\in H_{0}^{1}(\Omega), u_{1}\in L^{2}(\Omega)$, $a(u_{0},u_{0})> r_{*}^{2}$, $0<E(0)<M$ and
\begin{equation}\label{05.19}
  \int_{\Omega} u_{0}u_{1} \textrm{d}x  >\frac{p+1}{p+2}\left(  \frac{ B_{7}^{2}}{2\varepsilon' a_{0}\mu_{0}} \right)^{\frac{1}{p+1}}E(0),
\end{equation}
where $\varepsilon' \in (0,1)$ is a solution of the equation
\begin{eqnarray}\label{5.30}
\frac{p+1}{p+2}\left(  \frac{ B_{7}^{2}}{2\varepsilon'a_{0}\mu_{0}} \right)^{\frac{1}{p+1}}
=\frac{q(1-\varepsilon')}{2\sqrt{\frac{a_{0}\mu_{0}}{B_{7}}\left( 1+\frac{q(1-\varepsilon') }{2} \right)\left( \frac{q(1-\varepsilon') }{2}-1\right)}},
\end{eqnarray}
$B_{7}$ is the best constant of Sobolev embedding inequality $\|u\|_{2}^{2}\leq B_{7} \|\nabla u\|_{2}^{2}$. Then, the weak solution $u$ of \eqref{1.1} furnished by theorem \ref{theorem3.1} blows up in finite time.
\end{thm}

Now, let us give some properties of $J$, $I$ and $E$.

\begin{lem}\label{lemma2.2}
Let $u\in H_{0}^{1}(\Omega)\backslash\{0\}$, then we have

(i) $\lim\limits_{\lambda\rightarrow 0}J(\lambda u)=0$, $\lim\limits_{\lambda\rightarrow +\infty}J(\lambda u)=-\infty$;

(ii) There exists a unique $\lambda_{*}>0$, such that
\begin{equation}\label{2.15}
I(\lambda u)=\lambda \frac{d}{d\lambda}J(\lambda u)
\left\{\begin{array}{lll}
  &>0, &0<\lambda<\lambda_{*},\\
  &=0, &\lambda=\lambda_{*},\\
 &<0,  &\lambda>\lambda_{*}.
\end{array}\right.
\end{equation}
\end{lem}
\begin{proof}
For any $u\in H_{0}^{1}(\Omega)\backslash\{0\}$, we define
\begin{eqnarray*}
  j(\lambda):=J(\lambda u)&=&\frac{\lambda^{2}}{2}\mu(t)a(u(t),u(t))
  -\frac{\lambda^{q}}{q}\int_{\Omega}|u|^{q}\log(|\lambda u|)\textrm{d}x
  +\frac{\lambda^{q}}{q^{2}}\|u\|_{q}^{q}, \lambda>0.
\end{eqnarray*}
This map is introduced by Dr\'abek and Pohozaev \cite{D1997PROC}. By the definition of $I(u)$, we have
\begin{eqnarray*}
  j'(\lambda)&=&\lambda^{}\left( \mu(t)a(u(t),u(t))-\lambda^{q-2}\int_{\Omega}|u|^{q}\log(|u|)\textrm{d}x -\lambda^{q-2}\log(\lambda)\|u\|_{q}^{q} \right)\\
  &&=\frac{1}{\lambda}I(\lambda u).
\end{eqnarray*}
Let $m(\lambda)=\lambda^{-1} j'(\lambda)$. Taking the first derivative of $m(\lambda)$, we get
\begin{equation}\label{2.16}
  m'(\lambda)=-\lambda^{q-3}\left((q-2)\int_{\Omega}|u|^{q}\log(|u|)\textrm{d}x+ (q-2)\log(\lambda)\|u\|_{q}^{q} +\|u\|_{q}^{q} \right).
\end{equation}
It easily follows from \eqref{1.2} and \eqref{2.16} that there exists $0<\lambda_{1}<+\infty$ such that $m'(\lambda)>0$ if $0<\lambda<\lambda_{1}$; $m'(\lambda)<0$ if $\lambda>\lambda_{1}$. Thus $m(\lambda)$ is increasing in $(0,\lambda_{1})$ and decreasing in $(\lambda_{1}, +\infty)$. Since $m(0)=\mu(t)a(u(t),u(t))>0$ and $\lim\limits_{\lambda \rightarrow +\infty} m(\lambda)=-\infty$,  there exists $\lambda_{*}>0$ such that $m(\lambda_{*})=0$. Then $j'(\lambda_{*})=0$ and $j(\lambda)>0$ in $(0,\lambda_{*})$, $j(\lambda)<0$ in $(\lambda_{*}, +\infty)$. The proof  is complete.
\end{proof}

\begin{lem}\label{lemma2.4}
Let $u\in H^{1}_{0}(\Omega)\backslash\{0\}$, $\gamma\in (0,2^{*}-q)$ and
\begin{equation*}
  r(\gamma)=\left( \frac{\mu_{0}e\gamma}{K_{0}^{q+\gamma}} \right)^{\frac{1}{q+\gamma-2}},
\end{equation*}
where $K_{0}$ is a constant satisfying $\|u\|_{q+\gamma}\leq C_{q+\gamma}\|\nabla u\|_{2} \leq K_{0} [a(u,u)]^{\frac{1}{2}}$,
then we have

(i) if $0<a(u,u)<r^{2}(\gamma)$, then $I(u)>0$;

(ii) if $I(u)< 0$, then $a(u,u)>r^{2}(\gamma)$;

(iii) if $I(u)= 0$, then $a(u,u)\geq r^{2}(\gamma)$.
\end{lem}

\begin{proof}
By the basic inequality ($e\gamma \log x\leq x^{\gamma}$ for any $x>1, \gamma>0$) and Sobolev inequality, we have
\begin{eqnarray*}
I(u)&=&\mu(t)a(u(t),u(t))-\int_{\Omega}|u|^{q}\log(|u|)\textrm{d}x\\
&\geq& \mu_{0}a(u,u)-\frac{1}{e\gamma}\|u\|_{q+\gamma}^{q+\gamma}\\
&\geq& \mu_{0}a(u,u)-\frac{1}{e\gamma}K_{0}^{q+\gamma}\left(a(u,u)\right)^{\frac{q+\gamma}{2}}\\
&=&\frac{K_{0}^{q+\gamma}}{e\gamma}a(u,u)\left( \frac{\mu_{0}e\gamma}{K_{0}^{q+\gamma}}-\left(a(u,u)\right)^{\frac{q+\gamma-2}{2}}\right).
\end{eqnarray*}
By above inequality, we know that (i), (ii) and (iii) hold.
\end{proof}

\begin{lem}\label{lemma2.5}
For the notation in Lemma \ref{lemma2.4}, we have
\begin{eqnarray*}
  0<r_{*}=\sup_{\gamma\in (0,2^{*}-q)}\left( \frac{\mu_{0}e\gamma}{K_{0}^{q+\gamma}} \right)^{\frac{1}{q+\gamma-2}}
  \leq \rho_{*}=\sup_{\gamma\in (0,2^{*}-q)}\left( \frac{\mu_{0}e\gamma}{K_{1}^{q+\gamma}} \right)^{\frac{1}{q+\gamma-2}}|\Omega|^{\frac{\gamma}{q(q+\gamma-2)}}
  <+\infty,
\end{eqnarray*}
where
$$
K_{1}=\sup_{u\in H^{1}_{0}\backslash\{0\}} \frac{\|u\|_{q}}{[a(u,u)]^{\frac{1}{2}}}.
$$
\end{lem}
\begin{proof}
Obviously, if $r_{*}$ exists, then $r_{*}>0$. So, we just have to prove that $\rho_{*}$ exists, $r(\gamma)\leq\rho(\gamma)$ and $\rho_{*}<+\infty$, where
\begin{equation*}
  \rho(\gamma)=\left( \frac{\mu_{0}e\gamma}{K_{1}^{q+\gamma}} \right)^{\frac{1}{q+\gamma-2}}|\Omega|^{\frac{\gamma}{q(q+\gamma-2)}}.
\end{equation*}
For any $u\in H^{1}_{0}(\Omega)\backslash\{0\}$, by using the H\"older inequality, we have
$$
\|u\|_{q}\leq |\Omega|^{\frac{\gamma}{q(q+\gamma)}}\|u\|_{q+\gamma}.
$$
Hence, we obtain
\begin{equation*}
  K_{0}=\sup_{u\in H^{1}_{0}\backslash\{0\}} \frac{\|u\|_{q+\gamma}}{[a(u,u)]^{\frac{1}{2}}}
  \geq  \sup_{u\in H^{1}_{0}\backslash\{0\}} \frac{\|u\|_{q}}{[a(u,u)]^{\frac{1}{2}}}|\Omega|^{\frac{-\gamma}{q(q+\gamma)}}
  \geq|\Omega|^{\frac{-\gamma}{q(q+\gamma)}}K_{1},
\end{equation*}
then
\begin{equation*}
  \left( \frac{\mu_{0}e\gamma}{K_{0}^{q+\gamma}} \right)^{\frac{1}{q+\gamma-2}}
  \leq \left( \frac{\mu_{0}e\gamma}{K_{1}^{q+\gamma}} \right)^{\frac{1}{q+\gamma-2}}|\Omega|^{\frac{\gamma}{q(q+\gamma-2)}}.
\end{equation*}
That is, $r(\gamma)\leq\rho(\gamma)$.

Next, we prove $\rho_{*}$ exists and $\rho_{*}<+\infty$.

\textbf{Case 1:} If $n\geq 3$, then $\gamma\in (0,2^{*}-q)=(0,\frac{2n}{n-2}-q)$. Since $\rho(\gamma)$ is continuous on $[0,\frac{2n}{n-2}-q]$, we obtain that $\rho_{*}$ exists and
\begin{equation*}
  \rho_{*}=\sup_{\gamma\in(0,2^{*}-q)}\rho(\gamma)\leq \max_{\gamma\in[0,2^{*}-q]}\rho(\gamma)=\rho(\gamma_{*}) <+\infty.
\end{equation*}

\textbf{Case 2:} If $n=1, 2$, then $\gamma\in (0,+\infty)$. We define
\begin{eqnarray*}
  h(\gamma)&:=&\log\rho(\gamma)=\frac{1}{q+\gamma-2}\left( \log\gamma+\log(\mu_{0}e)-(q+\gamma)\log K_{1} \right)\\
  &&+\frac{\gamma}{q(q+\gamma-2)}|\Omega|, \ \ \gamma\in (0,\infty),
\end{eqnarray*}
then
\begin{equation*}
  h'(\gamma)=\frac{g(\gamma)}{q\gamma(q+\gamma-2)^{2}},
\end{equation*}
where
\begin{equation*}
  g(\gamma)=q^{2}+q\gamma-2q-q\gamma\log(\mu_{0}e)-q\gamma\log\gamma
  +2q\gamma\log K_{1}+q\gamma\log|\Omega|-2\gamma\log|\Omega|.
\end{equation*}
Taking the first derivative of $g(\gamma)$, we have
\begin{eqnarray*}
  g'(\gamma)&=&q-q\log(\mu_{0}e)-q\log\gamma-q+2q\log K_{1}+q\log|\Omega|-2\log|\Omega|\\
  &=&q\log\frac{K_{1}^{2}|\Omega|^{1-\frac{2}{q}}}{\mu_{0}e\gamma},
\end{eqnarray*}
which implies that $g(\gamma)$ is strictly increasing on $(0,\frac{K_{1}^{2}|\Omega|^{1-\frac{2}{q}}}{\mu_{0}e})$; strictly decreasing on $(\frac{K_{1}^{2}|\Omega|^{1-\frac{2}{q}}}{\mu_{0}e},+\infty)$.

Moreover, it is easy to see that
\begin{equation*}
\lim_{\gamma\rightarrow 0^{+}}g(\gamma)=q^{2}-2q>0,
\end{equation*}
 and
\begin{equation*}
 \lim_{\gamma\rightarrow +\infty}g(\gamma)=\lim_{\gamma\rightarrow +\infty}q^{2}-2q+q\gamma(1+\log(K_{1}^{2}|\Omega|^{1-\frac{2}{q}})-\log\gamma)=-\infty.
\end{equation*}
Then, there exists a unique $\gamma_{*}\in (\frac{K_{1}^{2}|\Omega|^{1-\frac{2}{q}}}{\mu_{0}e},+\infty)$ such that $g(\gamma_{*})=0$.

Hence, we have $g(\gamma)>0$ if $\gamma\in(0,\gamma_{*})$; $g(\gamma)<0$ if $\gamma\in (\gamma_{*},+\infty)$. Then, the maximum of $h(\gamma)$ can be reached at $\gamma=\gamma_{*}$. So,
\begin{equation*}
  \rho_{*}=\sup_{\gamma\in (0,+\infty)}\rho(\gamma)=e^{h(\gamma_{*})}<+\infty.
\end{equation*}
\end{proof}

\begin{corollary} \label{corollary2.1}
Let $u\in H^{1}_{0}(\Omega)\backslash\{0\}$. By combining Lemma \ref{lemma2.4} and Lemma \ref{lemma2.5}, we have

(i) if $0<a(u,u)<r_{*}^{2}$, then $I(u)>0$;

(ii) if $I(u)< 0$, then $a(u,u)> r_{*}^{2}$;

(iii) if $I(u)= 0$, then $a(u,u)\geq r_{*}^{2}$,\\
where $r_{*}$ is defined in Lemma \ref{lemma2.5}.
\end{corollary}

\begin{lem}\label{lemma2.6}
Assume that $q$ satisfies \eqref{1.2}. We obtain that
\begin{equation}\label{eq2.14}
  d\geq M:=\left( \frac{q-2}{2q} \right)\mu_{0}r_{*}^{2}>0,
\end{equation}
where $r_{*}$ is defined in Lemma \ref{lemma2.5}.
\end{lem}
\begin{proof}
According to \eqref{2.12}, we know $u\in \NN$ and $I(u)=0$. By the definition of $d$ and the Corollary \ref{corollary2.1}, we have
\begin{equation}\label{eq2.15}
  d\geq J(u)=\frac{1}{q^{2}}\|u\|_{q}^{q}+\frac{q-2}{2q}\mu(t)a(u(t),u(t))\geq \left( \frac{q-2}{2q} \right)\mu_{0}r_{*}^{2}, ~\text{for any}~ u\in \NN.
\end{equation}
\end{proof}

\begin{lem}\label{lemma2.1}
Assume that $({\bf{H2}})$ and $({\bf{H3}})$ hold, then $E(t)$ is a non-increasing function and
\begin{equation}\label{2.7}
  \frac{d}{dt}E(t)=\frac{1}{2}\mu'(t)a(u,u)-\int_{\Omega} g(u_{t}) u_{t}\textrm{d}x.
\end{equation}
\end{lem}

\begin{proof}
By direct computation,
\begin{eqnarray*}
\frac{\textrm{d}}{\textrm{d}t}E(t)
&=&\int_{\Omega}u_{tt}u_{t} \textrm{d}x +\frac{1}{2}\mu'(t)a(u,u)+\mu(t)\int_{\Omega }A(x)\nabla u  \cdot \nabla u_{t} \textrm{d}x\\
&&-\int_{\Omega}|u|^{q-2}u\log|u|u_{t}\textrm{d}x\\
&=&\int_{\Omega}u_{t}\left[u_{tt}-\mu(t) div(A(x) \nabla u)-|u|^{q-2}u\log|u|\right] \textrm{d}x+ \frac{1}{2}\mu'(t)a(u,u)\\
&=&\frac{1}{2}\mu'(t)a(u,u)-\int_{\Omega} g(u_{t}) u_{t} \textrm{d}x.
\end{eqnarray*}
It follows from ({\bf{H2}}) and Remark \ref{remark2.2} that $\frac{d}{dt}E(t) < 0$.
\end{proof}

\begin{rem}\label{remark2.1}
Integrating \eqref{2.7} from $0$ to $t$, we have
\begin{equation*}
  E(t)+\int_{0}^{t}\int_{\Omega} g(u_{t}) u_{t} \textrm{d}x\textrm{d}s - \frac{1}{2} \int_{0}^{t} \mu'(t) a(u,u) \textrm{d}s =E(0).
\end{equation*}
\end{rem}


Finally, we introduce the following lemma which plays  an essential  role in the proof of the energy decay.

\begin{lem}[\cite{NM1978}]\label{lemma4.1}
Let $\phi(t)$ be a non-increasing and nonnegative function on $[0,T], T>1$, such that
\begin{equation*}
  \phi(t)^{1+r}\leq w_{0}(\phi(t)-\phi(t+1)), \mbox{ on } [0,T],
\end{equation*}
where $w_{0}$ is a positive constant and $r$ is a nonnegative constant. Then we have

(i): if $r>0$, then
\begin{equation*}
  \phi(t)\leq (\phi(0)^{-r}+w_{0}^{-1}r[t-1]^{+})^{-\frac{1}{r}}, \mbox{ on } [0,T];
\end{equation*}

(i): if $r=0$, then
\begin{equation*}
  \phi(t)\leq \phi(0)e^{-w_{1}[t-1]^{+}}, \mbox{ on } [0,T];
\end{equation*}
where $w_{1}=\log \left( \frac{w_{0}}{w_{0}-1} \right)$,  $w_{0}>1$.
\end{lem}

\section{Local existence}
In this section, we always use $C$ to denote the positive constant, but it may take different values in different places.

In order to prove the Theorem \ref{theorem3.1}, as  \cite{GS2006}, for a given $T>0$, we introduce the space $\mathcal{M}=C([0,T];H_{0}^{1}(\Omega))\cap C^{1}([0,T]; L^{2}(\Omega))$ with the norm
\begin{equation*}
  \|u\|_{\mathcal{M}}^{2}=\max_{t\in [0,T]}\left( \|\nabla u\|_{2}^{2}+\|u_{t}\|_{2}^{2}  \right),
\end{equation*}
and first prove the following lemma.
\begin{lem} \label{lemma3.1}
For every $T>0, u\in \mathcal{M}$, $(u_{0},u_{1})\in H_{0}^{1}(\Omega)\times L^{2}(\Omega)$ and $c_{1}>\frac{1}{p+2}$, then there exists a unique function $v\in \mathcal{M}\cap C^{2}([0,T]; H^{-1}(\Omega))$ such that $v_{t}\in L^{p+2}([0,T],\Omega)\cap L^{2}([0,T]; H_{0}^{1}(\Omega))$ which solves the problem
\begin{equation}\label{3.2}
\left\{\begin{array}{llll}
&v_{tt}-\mu(t) Lv+g(v_{t})=|u|^{q-2}u\log|u|,& \mbox{ in } \Omega \times [0,T],\\
&v(x,t)=0,  &\mbox{ in } \partial\Omega \times [0,T],\\
&v(x,0)=u_{0}(x), v_{t}(x,0)=u_{1}(x),  &\mbox{ in } \Omega.&
\end{array}\right.
\end{equation}
\end{lem}
\begin{proof}
For convenience, we replace $v_{t}, v_{tt}$ with  $v', v''$. In what follows, we shall use the standard Galerkin method to prove this theorem.

Let $\{w_{j}\}_{j=1}^{\infty}$ is the orthogonal complete system of eigenfunctions of $-\Delta$ in $H_{0}^{1}(\Omega)$ with $\|w_{j}\|_{2}^{2}=1$. Then, $\{w_{j}\}_{j=1}^{\infty}$ is a complete orthonormal set in $L^{2}(\Omega)$ which is also the orthogonal basis of $H_{0}^{1}(\Omega)$, and $V_{n}$ is the finite dimensional space spanned by $\{ w_{1}, w_{2},...,w_{n}\}$. Let
\begin{equation*}
  u_{0k}=\sum_{j=1}^{k}(u_{0},w_{j})w_{j}, \ \ u_{1k}=\sum_{j=1}^{k}(u_{1},w_{j})w_{j}.
\end{equation*}
Obviously, $u_{0k}, u_{1k} \in V_{k}$, $u_{0k}\rightarrow u_{0}$  in $H_{0}^{1}(\Omega)$ and $u_{1k}\rightarrow u_{1}$  in $L^{2}(\Omega)$ as $k \rightarrow +\infty$. For all $k\geq 0$, we seek $k$ functions $r_{1k}, r_{2k},\cdots r_{kk} \in C^{2}[0,T]$ such that $v_{k}(t)=\sum_{j=1}^{k}r_{jk}(t)w_{j}$ solves the following problem
\begin{equation}\label{3.4}
\left\{\begin{array}{llll}
  \int_{\Omega}\left( v_{k}''+g(v_{k}')-|u|^{q-2}u\log|u|  \right)\eta \textrm{d}x +\mu(t)a(v_{k},\eta)=0, \\
  v_{k}(0)=u_{0k}, v_{k}'(0)=u_{1k},
\end{array}\right.
\end{equation}
for every $\eta\in V_{k}, t\geq 0$. For $j=1,2,\cdots k$, taking $\eta=w_{j}$, \eqref{3.4} is reduced to a nonlinear ordinary differential with unknown $r_{jk}$. By using standard existence theory for ordinary differential equations, we know that $r_{jk}(t)\in C^{2}[0,T]$ is a unique solution of the problem \eqref{3.4} for every $j$. Thus, we obtain a solution $v_{k}(t)$ of the equation \eqref{3.4} over $[0,T]$.

Multiplying \eqref{3.4} by $r_{jk}$ and summing over $j$ from $1$ to $k$, we get
\begin{eqnarray}\label{3.5}
&&\frac{\textrm{d}}{\textrm{d}t}\left( \frac{1}{2}\|v'_{k}\|_{2}^{2}+\frac{1}{2}\mu(t)a(v_{k},v_{k}) \right)+\int_{\Omega} g(v'_{k})v'_{k} \textrm{d}x\nonumber\\
&=&\frac{1}{2}\mu'(t)a(v_{k},v_{k})+\int_{\Omega}|u|^{q-2}u\log|u|v'_{k} \textrm{d}x.
\end{eqnarray}
For $\int_{\Omega}g(v'_{k})v'_{k} \textrm{d}x $, by ({\bf{H3}}) and Remark \ref{remark2.2}, we have
\begin{eqnarray}\label{3.6}
\int_{\Omega}g(v'_{k})v'_{k} \textrm{d}x
&=& \int_{\{x\in\Omega: |v'_{k}|\leq 1\}} g(v'_{k})v'_{k} \textrm{d}x+ \int_{\{x\in\Omega: |v'_{k}|> 1\}} g(v'_{k})v'_{k} \textrm{d}x\nonumber\\
&\geq& \int_{\{x\in\Omega: |v'_{k}|> 1\}} g(v'_{k})v'_{k} \textrm{d}x\nonumber\\
&\geq& c_{1}\int_{\Omega}|v'_{k}|^{p+2} \textrm{d}x -c_{1}\int_{\{x\in\Omega: |v'_{k}|\leq 1\}}|v'_{k}|^{p+2} \textrm{d}x \nonumber\\
&\geq& c_{1} \|v'_{k}\|_{p+2}^{p+2}-c_{1}|\Omega|.
\end{eqnarray}
Substituting \eqref{3.6} into \eqref{3.5} and integrating over $(0,t) \subset (0,T)$, we obtain
\begin{eqnarray}\label{3.7}
&&\|v'_{k}\|_{2}^{2}+\mu(t)a(v_{k},v_{k}) +2c_{1} \int_{0}^{t} \|v'_{k}\|_{p+2}^{p+2} \textrm{d}s \nonumber\\
&\leq& \|u_{1k}\|_{2}^{2}+\mu(0)a(u_{0k},u_{0k})+c_{1}|\Omega|T
+\|\mu'\|_{L^{\infty}}\int_{0}^{t}a(v_{k},v_{k}) \textrm{d}s\nonumber\\
&&+ 2\int_{0}^{t}\int_{\Omega} |u|^{q-2}u\log|u| v'_{k}  \textrm{d}s.\nonumber\\
\end{eqnarray}

By using H\"older's inequality, Young's inequality and the basic inequality ($|x^{q-1}\log x|<(e(q-1))^{-1}$ for $0<x<1$; $x^{-\nu}\log x<(e\nu)^{-1}$ for any $x\geq 1, \nu>0$), we can get
\begin{eqnarray}\label{3.8}
&&2\int_{0}^{t}\int_{\Omega} |u|^{q-2}u\log|u| v'_{k} \textrm{d}x \textrm{d}s\nonumber\\
&\leq& 2\int_{0}^{t}\left\| |u|^{q-1}\log|u|\right\|_{\frac{p+2}{p+1}}\left\|v'_{k}\right\|_{p+2} \textrm{d}s\nonumber\\
&\leq& \frac{2(p+1)}{p+1}\int_{0}^{t}\left\| |u|^{q-1}\log|u| \right\|^{\frac{p+2}{p+1}}_{\frac{p+2}{p+1}} \textrm{d}s +\frac{2}{p+2}\int_{0}^{t} \| v'_{k}\|_{p+2}^{p+2} \textrm{d}s \nonumber\\
&\leq& \frac{2(p+1)}{p+2}\int_{0}^{t} \left( (eq-e)^{-\frac{p+2}{p+1}} |\Omega| +(e\nu)^{-\frac{p+2}{p+1}}B_{1}^{\frac{(p+2)(q-1-\nu)}{p+1}} \| \nabla u \|_{2}^{\frac{(p+2)(q-1-\nu)}{p+1}} \right) \textrm{d}s \nonumber\\
&&+\frac{2}{p+2}\int_{0}^{t} \| v'_{k}\|_{p+2}^{p+2} \textrm{d}s\nonumber\\
&\leq & CT(1+\|u\|_{\mathcal{M}}^{\frac{(p+2)(q-1-\nu)}{p+1}}) + \frac{2}{p+2}\int_{0}^{t} \| v'_{k}\|_{p+2}^{p+2} \textrm{d}s\nonumber\\
&\leq & C + \frac{2}{p+2}\int_{0}^{t} \| v'_{k}\|_{p+2}^{p+2} \textrm{d}s,
\end{eqnarray}
where $B_{1}$ is a positive embedding constant. Combining \eqref{3.7} and \eqref{3.8}, we have
\begin{eqnarray*}
&&\|v'_{k}\|_{2}^{2}+\mu(t)a(v_{k},v_{k}) +2(c_{1} -\frac{1}{p+2})\int_{0}^{t} \|v'_{k}\|_{p+2}^{p+2}  \textrm{d}s \nonumber\\
&\leq& C+\|u_{1k}\|_{2}^{2}+\mu(0)a(u_{0k},u_{0k})+c_{1}|\Omega|T
+\|\mu'\|_{L^{\infty}}\int_{0}^{t}a(v_{k},v_{k}) \textrm{d}s .
\end{eqnarray*}
Thus, by Gronwall's lemma, we obtain
\begin{eqnarray}\label{3.10}
\|v'_{k}\|_{2}^{2}+ a(v_{k},v_{k}) +2c_{1} \int_{0}^{t}  \|v'_{k}\|_{p+2}^{p+2}  \textrm{d}s \leq C.
\end{eqnarray}
From \eqref{2.5} and \eqref{3.10}, it follows that
\begin{eqnarray}\label{3.11}
  &\{v_{k}\} &\mbox{ is bounded in } L^{\infty}(0,T;H_{0}^{1}(\Omega)),\nonumber\\
  &\{v_{k}'\} &\mbox{ is bounded in } L^{\infty}(0,T;L^{2}(\Omega)),\nonumber\\
  & |v'_{k}|^{p}v'_{k} &\mbox{ is bounded in } L^{\infty}(0,T;L^{\frac{p+2}{p+1}}(\Omega)).
\end{eqnarray}
Therefore, by ({\bf{H3}}) and \eqref{3.11}, there exists a subsequence of $\{v_{k}\}$, which we still denote by itself, such that
\begin{eqnarray*}
  &v_{k} \rightharpoonup v  &\mbox{ weakly star in } L^{\infty}(0,T;H_{0}^{1}(\Omega)),\nonumber\\
  &v'_{k} \rightharpoonup v'  &\mbox{ weakly star in } L^{\infty}(0,T;L^{2}(\Omega)),\nonumber\\
  &v''_{k} \rightharpoonup v''  &\mbox{ weakly star in } L^{\infty}(0,T;H_{0}^{-1}(\Omega)),\nonumber\\
  &g(v'_{k}) \rightharpoonup \chi  &\mbox{ weakly star in } L^{\infty}(0,T;L^{\frac{p+2}{p+1}}(\Omega)).
\end{eqnarray*}
The rest of the proof of this lemma is the same as Proposition 3.1 in \cite{HAMO2021}, so we omit it. Therefore, up to subsequence, we may pass to the limit in \eqref{3.4} and get $\chi=g(v')$. Then the existence of a local weak solution of \eqref{3.2} is finished.

Now, we prove the uniqueness of the solution. If $v_{1}$ and $v_{2}$ are two solutions of \eqref{3.2} with the same initial data and set $w=v_{1}-v_{2}$, we have
\begin{equation}\label{3.13}
  \int_{\Omega}w'' \phi \textrm{d}x +\mu(t)a(w, \phi)+ \int_{\Omega}\left(g(v'_{1})-g(v'_{2})\right)\phi \textrm{d}x=0.
\end{equation}
Let $\phi=w'$ in \eqref{3.13} and by the fact that $g$ is monotonously nondecreasing, then
\begin{equation}\label{3.14}
  \frac{\textrm{d}}{\textrm{d}t}\left(\|w'\|_{2}^{2}+\mu(t)a(w,w)\right)\leq C a(w,w).
\end{equation}
By using Gronwall' lemma, \eqref{3.14} yields $\|w'\|_{2}^{2}=a(w,w)=0$, that is , $v_{1}=v_{2}$. The proof of the lemma is complete.
\end{proof}

\noindent\textbf{\emph{Proof of Theorem \ref{theorem3.1}.}}
Let $R^{2}=2(\| \nabla u_{0}\|_{2}^{2}+\|u_{1}\|_{2}^{2})$ and for any $T>0$, we define
\begin{equation*}
  \mathcal{W}_{T}=\{ u\in \mathcal{M}:u(0)=u_{0}, u_{t}(0)=u_{1} \mbox{ and } \|u\|_{\mathcal{M}}\leq R  \}.
\end{equation*}
By Lemma \ref{lemma3.1}, for any $u \in \mathcal{W}_{T}$, we define $v=\Phi(u)$, where $v$ is the unique solution of \eqref{3.2}. Now, we claim that for a suitable $T>0$, $\Phi$ is a contractive map with $\Phi(\mathcal{W}_{T})\subset \mathcal{W}_{T}$. Precisely, given $u \in \mathcal{W}_{T}$, the corresponding solution $v=\Phi(u)$ satisfies the following identify
\begin{eqnarray}\label{3.16}
  &&\|v_{t}\|_{2}^{2}+\mu(t)a(v,v)+2\int_{0}^{t}\int_{\Omega} g(v')v' \textrm{d}x \textrm{d}s \nonumber\\
   &=& \|u_{1}\|_{2}^{2}+\mu(0)a(u_{0},u_{0})+ 2\int_{0}^{t}\int_{\Omega}|u|^{q-2}u\log|u| v' \textrm{d}x \textrm{d}s, \ \ \mbox{for all} \ \ t\in [0,T].\nonumber\\
\end{eqnarray}
From \eqref{3.6} and \eqref{3.8}, we have
\begin{eqnarray*}\label{3.17}
  \int_{\Omega}g(v')v' \textrm{d}x
&\geq& c_{1} \|v'\|_{p+2}^{p+2}-c_{1}|\Omega|, \nonumber\\
  2\int_{0}^{t}\int_{\Omega}|u|^{q-2}u\log|u| v' \textrm{d}x \textrm{d}s &\leq& CT(1+R^{\frac{(p+2)(q-1-\nu)}{p+1}}) + \frac{2}{p+2}\int_{0}^{t} \| v'\|_{p+2}^{p+2} \textrm{d}s.\nonumber\\
\end{eqnarray*}
Then, \eqref{3.16} can be reduced to
\begin{equation*}\label{3.18}
  \|v_{t}\|_{2}^{2}+\mu(t)a(v,v) \leq R^{2}+CT(1+R^{\frac{(p+2)(q-1-\nu)}{p+1}}),
\end{equation*}
for all $t\in (0,T]$. Choosing $T$ sufficiently small, we get $\|v\|_{\mathcal{M}} \leq R$ which shows that $\Phi(\mathcal{W}_{T}) \subset \mathcal{W}_{T}$.

Next, we show that $\Phi$ is contractive in $\mathcal{W}_{T}$. That is, for any $u_{1}, u_{2} \in \mathcal{W}_{T}$, there exists $0<\delta<1$ such that
\begin{equation*}\label{3.19}
  \|\Phi(u_{1}) -\Phi(u_{2})\|_{\mathcal{M}}\leq \delta \|u_{1}-u_{2}\|_{\mathcal{M}}.
\end{equation*}
Let $u_{1}, u_{2}\in \mathcal{W}_{T}, v_{1}=\Phi(u_{1}), v_{2}=\Phi(u_{2})$ and $z=v_{1}-v_{2}$. Then, subtracting the two equations \eqref{3.2} with $u_{1}$ and $u_{2}$, we obtain
\begin{eqnarray}\label{3.20}
&&\int_{\Omega} z''\phi \textrm{d}x +\mu(t)\int_{\Omega}A(x)\nabla z \cdot \nabla \phi \textrm{d}x +\int_{\Omega}\left( g(v_{1}')-g(v_{2}') \right)\phi \textrm{d}x \nonumber\\
&=&\int_{\Omega}\left( |u_{1}|^{q-2}u_{1}\log|u_{1}|- |u_{2}|^{q-2}u_{2}\log|u_{2}| \right)\phi \textrm{d}x,
\end{eqnarray}
for all $\phi \in H^{1}_{0}(\Omega)$ and a.e. $t\in [0,T]$ with $z(x,0)=0, \ \ z'(x,0)=0$.

Taking $\phi=z'$ and integrating \eqref{3.20} over $(0,t)\subset [0,T]$, then we have
\begin{eqnarray}\label{3.21}
  \| z'\|_{2}^{2}+\mu(t)a(z,z)
  \leq 2\int_{0}^{t}\int_{\Omega}\left( |u_{1}|^{q-2}u_{1}\log|u_{1}|- |u_{2}|^{q-2}u_{2}\log|u_{2}| \right)z' \textrm{d}x \textrm{d}s.
\end{eqnarray}
Now, let's estimate the right-hand side of \eqref{3.21}. Using Lagrange's mean value theorem, there exists $\theta \in (0,1)$ such that
\begin{eqnarray}\label{3.22}
 &&\Big| |u_{1}|^{q-2}u_{1}\log|u_{1}|- |u_{2}|^{q-2}u_{2}\log|u_{2}|\Big|\nonumber\\
 &=&\Big| 1+(q-1)\log|\theta u_{1}+(1-\theta) u_{2}| \Big|\cdot \Big| \theta u_{1}+(1-\theta) u_{2} \Big|^{q-2}\cdot |u_{1}-u_{2}|\nonumber\\
 &\leq& \left| \theta u_{1}+(1-\theta) u_{2} \right|^{q-2}\cdot |u_{1}-u_{2}| +(q-1)A_{1}|u_{1}-u_{2}|\nonumber\\
&& +(q-1)\left| \theta u_{1}+(1-\theta) u_{2} \right|^{q-2+\epsilon}\cdot |u_{1}-u_{2}|\nonumber\\
 &\leq&|u_{1}+u_{2}|^{q-2}|u_{1}-u_{2}|+(q-1)A_{1}|u_{1}-u_{2}|
 +(q-1)\left| u_{1}+ u_{2} \right|^{q-2+\epsilon}\cdot |u_{1}-u_{2}|,\nonumber\\
\end{eqnarray}
where the second inequality utilizes the Lemma 2.1 in \cite{KM2020}. That is, for every $\epsilon>0$, there exists $A_{1}>0$ such that the real function $f(s)=|s|^{q-2}\log|s|~(p>2)$ satisfies $|f(s)|\leq A_{1}+|s|^{q-2+\epsilon}$.

Since $u_{1}, u_{2} \in \mathcal{W}_{T}$, using H\"older's inequality and Sobolev embedding theorem, then
\begin{eqnarray}\label{3.23}
&&\int_{\Omega}\left| |u_{1}+u_{2}|^{q-2}|u_{1}-u_{2}|\right|^{2} \textrm{d}x\nonumber\\
&\leq&\left( \int_{\Omega}|u_{1}+u_{2}|^{2(q-1)} \textrm{d}x \right)^{\frac{q-2}{q-1}} \left( \int_{\Omega}|u_{1}-u_{2}|^{2(q-1)} \textrm{d}x \right)^{\frac{1}{q-1}}\nonumber\\
&\leq&C \left( \|u_{1}\|_{2(q-1)}^{2(q-1)}+ \|u_{2}\|_{2(q-1)}^{2(q-1)} \right)^{\frac{q-2}{q-1}}\|u_{1}-u_{2}\|^{2}_{2(q-1)}\nonumber\\
&\leq&C \left( \|u_{1}\|_{\mathcal{M}}^{2(q-1)}+ \|u_{2}\|_{\mathcal{M}}^{2(q-1)} \right)^{\frac{q-2}{q-1}}\|u_{1}-u_{2}\|^{2}_{\mathcal{M}}\nonumber\\
&\leq& CR^{2(q-2)}\|u_{1}-u_{2}\|_{\mathcal{M}}^{2},
\end{eqnarray}
and
\begin{eqnarray}\label{3.24}
&&\int_{\Omega}\left| | u_{1}+ u_{2} |^{q-2+\epsilon}\cdot |u_{1}-u_{2}| \right|^{2} \textrm{d}x\nonumber\\
&\leq&\left( \int_{\Omega}|u_{1}+u_{2}|^{\frac{2(q-2+\epsilon)(q-1)}{q-2}}  \textrm{d}x \right)^{\frac{q-2}{q-1}} \left( \int_{\Omega}|u_{1}-u_{2}|^{2(q-1)} \textrm{d}x \right)^{\frac{1}{q-1}}\nonumber\\
&=&\left( \int_{\Omega}|u_{1}+u_{2}|^{2(q-1)+\frac{2\epsilon (q-1)}{q-2}}  \textrm{d}x \right)^{\frac{q-2}{q-1}}\|u_{1}-u_{2}\|_{2(q-1)}^{2}.
\end{eqnarray}

By $q<\frac{2n}{n-2}$, we can choose $\epsilon$ small enough such that $2(q-1)+\frac{2\epsilon (q-1)}{q-2}<\frac{2(n+2)}{n-2}$. So, we have
\begin{eqnarray*}
\int_{\Omega}|u_{1}+u_{2}|^{2(q-1)+\frac{2\epsilon (q-1)}{q-2}}  \textrm{d}x &\leq& \int_{\Omega}|u_{1}|^{2(q-1)+\frac{2\epsilon (q-1)}{q-2}}  \textrm{d}x+\int_{\Omega}|u_{2}|^{2(q-1)+\frac{2\epsilon (q-1)}{q-2}}  \textrm{d}x\nonumber\\
&\leq& C\left( \|\nabla u_{1}\|_{2}^{\frac{2(n+2)}{n-2}}+\|\nabla u_{2}\|_{2}^{\frac{2(n+2)}{n-2}} \right).
\end{eqnarray*}
 And \eqref{3.24} can be reduced to
\begin{eqnarray}\label{03.24}
\int_{\Omega}\left| | u_{1}+ u_{2} |^{q-2+\epsilon}\cdot |u_{1}-u_{2}| \right|^{2} \textrm{d}x
\leq CR^{\frac{2(n+2)(q-2)}{(n-2)(q-1)}}\|u_{1}-u_{2}\|_{\mathcal{M}}^{2}.
\end{eqnarray}
Thus, from \eqref{3.21}--\eqref{03.24}, we have
\begin{eqnarray*}\label{3.25}
  \| z'\|_{2}^{2}+\mu(t)a(z,z)+2\int_{0}^{t} \|\nabla z'\|_{2}^{2} \textrm{d}s
   \leq CT(1+R^{2(q-2)}+R^{\frac{2(n+2)(q-2)}{(n-2)(q-1)}})\|u_{1}-u_{2}\|_{\mathcal{M}}^{2}.
\end{eqnarray*}
If $T$ is sufficiently small, then there exists $\delta \in(0,1)$ such that
\begin{equation*}\label{3.26}
  \|\Phi(u_{1})-\Phi(u_{2})\|_{\mathcal{M}}^{2}=\|v_{1}-v_{2}\|_{\mathcal{M}}^{2}\leq \delta \|u_{1}-u_{2}\|_{\mathcal{M}}^{2}.
\end{equation*}
So, by the contraction mapping principle, there exists a unique weak solution for \eqref{1.1} in $[0,T]$.
$\hfill\qedsymbol$

\section{Global existence and energy decay}

In this section, we prove that the weak solution provided by theorem \ref{theorem3.1} can be extended to the whole interval $[0,+\infty)$ and give the accurate energy decay rates with the damping having a general growth near zero.

\begin{lem}\label{lemma2.3}
Assume that the conditions of Theorem \ref{theorem3.1} hold, $u$ is a local weak solution for problem \eqref{1.1}. If $E(0)<M$, $a(u_{0},u_{0})<r_{*}^{2}$ and $u_{1}\in L^{2}(\Omega)$, then $a(u(t),u(t))<r_{*}^{2}$ for all $t\in [0,T]$.
\end{lem}
\begin{proof}
If the conclusion is false, we assume that $t_{0}=\inf\{ \tau>0: a(u(\tau),u(\tau))=r_{*}^{2} \}$. Then,  $a(u(t),u(t))<r_{*}^{2}$ over $t \in [0,t_{0})$ and  by continuity $a(u(t_{0}),u(t_{0}))=r_{*}^{2}$. From \eqref{eq2.15}, we have
$$
E(t_{0})\geq J(u(t_{0}))\geq M=\left( \frac{q-2}{2q} \right)\mu_{0}r_{*}^{2},
$$
which is impossible. Since by Lemma \ref{lemma2.1}, we have
$$
E(t)\leq E(0)<M=\left( \frac{q-2}{2q} \right)\mu_{0}r_{*}^{2}, ~\text{for all} ~t\in [0,T].
$$
Then, $a(u(t),u(t))<r_{*}^{2}$ for all $t\in [0,T]$.
\end{proof}

\begin{lem}\label{lemma2.7}
By Corollary \ref{corollary2.1} and Lemma \ref{lemma2.3}, we can see that if $E(0)<M$ and $a(u_{0},u_{0})<r_{*}^{2}$ then $J(t)>0$ and $E(t)>0$ for all $t\in [0,T]$.
\end{lem}
\begin{proof}
By Lemma \ref{lemma2.3}, we obtain $a(u(t),u(t))<r_{*}^{2}$ for all $t\in [0,T]$. And by Corollary \ref{corollary2.1}, we have $I(u(t))>0$. Then, according to the definitions of $E(t)$ and $J(u)$, we obtain $J(t)>0$ and $E(t)>0$ for all $t\in [0,T]$.
\end{proof}

\subsection{Proof of Theorem \ref{theorem3.2}}

\textbf{case 1: $q<p+2.$}

Here, we define a now functional
\begin{equation*}
  e(t)=\frac{1}{2}\|u_{t}\|_{2}^{2}+\frac{1}{2}\mu(t)a(u(t),u(t))
  +\frac{1}{q^{2}}\|u\|_{q}^{q}.
\end{equation*}
Then, by Lemma \ref{lemma2.1}, we have
\begin{eqnarray}\label{e4.1}
  \frac{\textrm{d}}{\textrm{d}t}e(t)+\int_{\Omega} g(u_{t})u_{t} \textrm{d}x
  = \frac{1}{2}\mu'(t)a(u,u)+\frac{1}{q}\int_{\Omega} |u_{t}|^{q-1} u_{t}\textrm{d}x+\int_{\Omega}|u|^{q-1}\log |u|u_{t} \textrm{d}x.
\end{eqnarray}
Now, we estimate some terms in \eqref{e4.1}. By \eqref{3.6}, we can directly obtain the following  inequalities,
\begin{equation}\label{e4.2}
  \int_{\Omega}g(u_{t})u_{t} \textrm{d}x \geq c_{1}\|u_{t}\|_{p+2}^{p+2}-c_{1}|\Omega|.
\end{equation}
By the condition $q<p+2$, let
\begin{equation*}
  \nu:=q\cdot\frac{p+1}{p+2}-q+1>0.
\end{equation*}
Then, by \eqref{3.8} and Young's inequality, we have
\begin{eqnarray}\label{e4.3}
&&\int_{\Omega} |u|^{q-2}u\log|u| u_{t}  \textrm{d}x \nonumber\\
&\leq& \frac{1}{p+2}\|u_{t}\|_{p+2}^{p+2}+\frac{p+1}{p+2} \int_{\Omega}\left| |u|^{q-2}u\log|u| \right|^{\frac{p+2}{p+1}}\textrm{d}x \nonumber\\
&\leq& \frac{1}{p+2}\|u_{t}\|_{p+2}^{p+2}+\frac{p+1}{p+2}  \int_{x\in \Omega: |u|<1}\left| |u|^{q-2}u\log|u| \right|^{\frac{p+2}{p+1}}\textrm{d}x\nonumber\\
&&+\frac{p+1}{p+2}  \int_{x\in \Omega: |u|\geq1}\left| |u|^{q-2}u\log|u| \right|^{\frac{p+2}{p+1}}\textrm{d}x \nonumber\\
&\leq& \frac{1}{p+2}\|u_{t}\|_{p+2}^{p+2}+\frac{p+1}{p+2}(eq-e)^{-\frac{p+2}{p+1}}|\Omega| +\frac{p+1}{p+2}(e\nu)^{-2}\int_{\Omega}|u|^{\frac{p+2}{p+1}(q-1+\nu)} \textrm{d}x\nonumber\\
&=&\frac{1}{p+2}\|u_{t}\|_{p+2}^{p+2}+\frac{p+1}{p+2}(eq-e)^{-\frac{p+2}{p+1}}|\Omega| +\frac{p+1}{p+2}(e\nu)^{-2}\|u\|_{q}^{q}.
\end{eqnarray}
Using Young's inequality and H\"older's inequality with $\frac{q-1}{q}+\frac{1}{q}=1$, the imbedding $L^{p+2}(\Omega)\hookrightarrow L^{q}(\Omega)$ with imbedding constant $C_{p+2}$ and $\delta>0$, we can obtain
\begin{eqnarray}\label{e4.4}
\int_{\Omega} |u_{t}|^{q-1} u_{t}\textrm{d}x
&\leq& \delta^{-\frac{1}{q-1}}\|u\|_{q}^{q}+\delta\|u_{t}\|_{q}^{q}\nonumber\\
&\leq& \delta^{-\frac{1}{q-1}}\|u\|_{q}^{q}+\delta C_{p+2}^{q}\|u_{t}\|_{p+2}^{q}\nonumber\\
&\leq& \delta^{-\frac{1}{q-1}}\|u\|_{q}^{q}+\delta C_{p+2}^{q}2^{p+1}(1+\|u_{t}\|_{p+2}^{p+2}).
\end{eqnarray}
Substituting \eqref{e4.2}--\eqref{e4.4} into \eqref{e4.1}, then
\begin{eqnarray}\label{e4.5}
  &&\frac{\textrm{d}}{\textrm{d}t}e(t)\nonumber\\
  &\leq& -c_{1}\|u_{t}\|_{p+2}^{p+2}+c_{1}|\Omega|+\frac{1}{p+2}\|u_{t}\|_{2}^{2}
  +\frac{p+1}{p+2}(eq-e)^{-\frac{p+2}{p+1}}|\Omega| \nonumber\\ &&+\frac{p+1}{p+2}(e\nu)^{-2}\int_{\Omega}|u|^{q} \textrm{d}x
  +\frac{\delta^{-\frac{1}{q-1}}}{q}\|u\|_{q}^{q}+\frac{\delta}{q} C_{p+2}^{q}2^{p+1}(1+\|u_{t}\|_{p+2}^{p+2})\nonumber\\
  &=&\left(\frac{p+1}{p+2}(e\nu)^{-2}+\frac{\delta^{-\frac{1}{q-1}}}{q}\right)\|u\|_{q}^{q}
  +\left(\frac{1}{p+2}+\frac{\delta}{q} C_{p+2}^{q}2^{p+1}-c_{1}\right)\|u_{t}\|_{p+2}^{p+2}\nonumber\\
  &&+\left( c_{1}|\Omega|+ \frac{p+1}{p+2}(eq-e)^{-\frac{p+2}{p+1}}|\Omega|+\delta C_{p+2}^{q}2^{p+1}\right).
\end{eqnarray}
Taking $\delta=\frac{q(c_{1}-\frac{1}{p+2})}{C_{p+2}^{q}2^{p+2}}>0$(see the conditions of Theorem \ref{theorem3.1}), we get
\begin{eqnarray}\label{e4.6}
  \frac{\textrm{d}}{\textrm{d}t}e(t)\leq C_{1}+C_{2}e(t).
\end{eqnarray}
And utilizing Gronwall inequality, we obtain
\begin{equation*}
  e(t)\leq (C_{3}e(0)+C_{4})e^{C_{3}t},
\end{equation*}
where $C_{i}(i=1,2,3,4)$ are positive constants. Then, as in \cite{JDE1994}, from this estimate and the continuation principle we finish the proof.

\textbf{case 2: $(u_{0},u_{1})\in W$.}

%
%

Revisit the potential energy functional $J(u)$,
\begin{equation}\label{3.27}
M>E(0)\geq E(t)\geq J(u)=\frac{1}{2}\mu(t)a(u(t),u(t))-\frac{1}{q}\int_{\Omega}|u|^{q}\log |u| dx+\frac{1}{q^{2}}\|u\|_{q}^{q}.
\end{equation}
For the logarithmic term, as in \eqref{3.8}, we choose $\mu>0$ sufficiently small such that $q+\mu<2^{*}$. Then, by direct calculation, we have
\begin{eqnarray}\label{3.28}
\int_{\Omega}|u|^{q}\log(|u|)\textrm{d}x
&=&\int_{\{x\in\Omega:|u|<1\}}|u|^{q}\log(|u|)\textrm{d}x
+\int_{\{x\in\Omega:|u|\geq 1\}}|u|^{q}\log(|u|)\textrm{d}x\nonumber\\
&\leq& e(q-1)^{-1}|\Omega|+(e\mu)^{-1}\int_{\{x\in\Omega:|u|\geq 1\}}|u|^{(q+\mu)}\textrm{d}x\nonumber\\
&\leq& e(q-1)^{-1}|\Omega|+(e\mu)^{-1}B_{1}^{(q+\mu)}\|\nabla u\|_{2}^{(q+\mu)},
\end{eqnarray}
where $B_{1}$ is the best constant of Sobolev embedding $H^{1}_{0}(\Omega)\hookrightarrow L^{q+\mu}(\Omega)$. Furthermore, it follows from Corollary \ref{corollary2.1} that $I(u)>0$. Combining ({\bf{H2}}), \eqref{2.5} and $J(u)<M$, we have
\begin{equation}\label{3.30}
  \|\nabla u\|_{2}^{2}\leq \frac{2q}{\mu_{0}a_{0}(q-2)}E(0)< \frac{2q}{\mu_{0}a_{0}(q-2)}M.
\end{equation}

Substituting \eqref{3.28}--\eqref{3.30} into \eqref{3.27}, we get
\begin{eqnarray*}\label{3.29}
&&\frac{1}{2} \|u_{t}\|_{2}^{2} +\frac{1}{2}\mu(t)a(u(t),u(t))+\frac{1}{q^{2}}\|u\|_{q}^{q}\nonumber\\
&\leq& \frac{1}{2} \|u_{t}\|_{2}^{2}+J(u)+e(q-1)^{-1}|\Omega|+(e\mu)^{-1}B_{1}^{(q+\mu)}\|\nabla u\|_{2}^{(q+\mu)}\nonumber\\
&<&M+e(q-1)^{-1}|\Omega|+(e\mu)^{-1}B_{1}^{(q+\mu)}\left( \frac{2q}{\mu_{0}a_{0}(q-2)}M\right)^{(q+\mu)}\nonumber\\
&:=&C,
\end{eqnarray*}
where the constant $C$ is independent of $t$. The continuation principle imply the global existence, i.e. $T=+\infty$.
$\hfill\qedsymbol$

\subsection{Proof of Theorem \ref{theorem4.2}}

Here, we give the energy decay rates when the damping has a general growth near zero. Note that Lemma \ref{lemma4.1} plays an essential role in this subsection.

Firstly, we claim that there exists a constant $\theta \in (0,1)$ such that
\begin{equation}\label{4.3}
  I(u)\geq \theta \mu(t)a(u,u).
\end{equation}
In fact, by Lemma \ref{lemma2.3} and Corollary \ref{corollary2.1}, we have $I(u(t))>0$ for all $t\geq 0$. Then, by Lemma \ref{lemma2.2}, there exists a constant $\lambda_{0}>1$ such that $I(\lambda_{0} u(t))=0$ and
\begin{eqnarray*}
M\leq J(\lambda_{0}u)&=&\frac{1}{q}I(\lambda_{0}u)+\frac{1}{q^{2}}\|\lambda_{0}u\|_{q}^{q}
+\frac{q-2}{2q}\mu(t)a(\lambda_{0}u(t),\lambda_{0}u(t))\nonumber\\
&=&\frac{1}{q^{2}} \lambda_{0}^{q}\|u\|_{q}^{q}+\frac{q-2}{2q}\lambda_{0}^{2}\mu(t)a(u(t),u(t))\nonumber\\
&\leq&\lambda_{0}^{q}\left( \frac{1}{q^{2}}\|u\|_{q}^{q}+\frac{q-2}{2q}\mu(t)a(u(t),u(t))  \right)\nonumber\\
&<&\lambda_{0}^{q}E(0),
\end{eqnarray*}
which means that
\begin{equation}\label{4.4}
  \lambda_{0}>\left( \frac{M}{E(0)} \right)^{\frac{1}{q}} >1.
\end{equation}
By the definition of $I(u)$, we have
\begin{eqnarray}\label{4.5}
0=I(\lambda_{0}u)&=&\mu(t)a(\lambda_{0}u(t),\lambda_{0}u(t))
-\int_{\Omega}|\lambda_{0}u|^{q}\log(|\lambda_{0}u|)\textrm{d}x\nonumber\\
&=&\lambda_{0}^{2}\mu(t)a(u,u)-\lambda_{0}^{q}\int_{\Omega}|u|^{q}\log|u| \textrm{d}x -(\lambda_{0}^{q}\log \lambda_{0})\|u\|_{q}^{q}\nonumber\\
&=&\lambda_{0}^{q}I(u)-\lambda_{0}^{q}\mu(t)a(u,u)+\lambda_{0}^{2}\mu(t)a(u,u)
-(\lambda_{0}^{q}\log \lambda_{0})\|u\|_{q}^{q}\nonumber\\
&=&\lambda_{0}^{q}I(u)-(\lambda_{0}^{q}-\lambda_{0}^{2})\mu(t)a(u,u)-
(\lambda_{0}^{q}\log \lambda_{0})\|u\|_{q}^{q}.\nonumber\\
\end{eqnarray}
Substituting \eqref{4.4} into \eqref{4.5}, we get
\begin{equation*}
  I(u)\geq (1-\lambda_{0}^{2-q})\mu(t)a(u,u).
\end{equation*}
Then the inequality \eqref{4.3} holds with $\theta=1-\lambda_{0}^{2-q}$.

Now, we prove the mainly result. By that assumption of $\beta$, we can rewrite ({\bf{H3}}) as follows:
\begin{eqnarray}\label{04.7}
  & |s|^{p+1}\leq |g(s)|\leq |s|^{\frac{1}{p+1}} &\mbox{ if } |s|\leq 1,\nonumber\\
  &c_{1}|s|^{p+1}\leq |g(s)| \leq c_{2}|s|^{p+1} &\mbox{ if }|s|>1.
\end{eqnarray}
Furthermore, by reviewing Remark \ref{remark2.1}, for all $t>0$, we define
\begin{equation}\label{4.6}
  E(t)-E(t+1)=D^{p+2}(t),
\end{equation}
where
\begin{equation*}
  D(t)^{p+2}=\int_{t}^{t+1}\int_{\Omega} g(u_{\tau}) u_{\tau} \textrm{d}x\textrm{d}\tau - \frac{1}{2} \int_{0}^{t} \mu'(\tau) a(u,u) \textrm{d}\tau >0.
\end{equation*}
From \eqref{04.7} and\eqref{4.6}, it is easy to get
\begin{eqnarray}\label{4.8}
  \int_{t}^{t+1}\int_{\Omega}|u_{t}|^{2} \textrm{d}x\textrm{d}t
  &=& \int_{t}^{t+1}\int_{\Omega}\left( |u_{t}|^{p+2}\right)^{\frac{2}{p+2}} \textrm{d}x\textrm{d}t \nonumber\\
  &\leq& \max\{c_{1}^{-\frac{2}{p+2}},1\}\int_{t}^{t+1}\int_{\Omega}\left( g(u_{t})u_{t}\right)^{\frac{2}{p+2}} \textrm{d}x\textrm{d}t \nonumber\\
    &\leq&\max\{c_{1}^{-\frac{2}{p+2}},1\}|\Omega|^{\frac{p}{p+2}}\left( \int_{t}^{t+1} \int_{\Omega} g(u_{t})u_{t} \textrm{d}x\textrm{d}t \right)^{\frac{2}{p+2}}  \nonumber\\
    &\leq& B_{2} D^{2}(t),
\end{eqnarray}
where $B_{2}=\max\{c_{1}^{-\frac{2}{p+2}},1\}|\Omega|^{\frac{p}{p+2}}>0$. Thus, there exist $t_{1}\in [t,t+\frac{1}{4}]$ and $t_{2}\in[t+\frac{3}{4},t+1]$ such that
\begin{equation}\label{4.9}
  \|u_{t}(t_{i})\|_{2}^{2}\leq 4B_{2}D^{2}(t), \  \  i=1,2.
\end{equation}
Multiplying \eqref{1.1} by $u$ and integrating over $\Omega\times[t_{1},t_{2}]$, we have
\begin{eqnarray}\label{4.10}
\int_{t_{1}}^{t_{2}} I(u) \textrm{d}t
&=&\int_{t_{1}}^{t_{2}}\|u_{t}\|_{2}^{2} \textrm{d}t +\int_{\Omega} \left(u_{t}(t_{1})u(t_{1})-u_{t}(t_{2})u(t_{2}) \right) \textrm{d}x\nonumber\\
&&-\int_{t_{1}}^{t_{2}}\int_{\Omega} g(u_{t})u \textrm{d}x\textrm{d}t.
\end{eqnarray}
Now, we estimate every term on the right hand side of \eqref{4.10}. It follows from \eqref{3.30}, \eqref{4.8} and \eqref{4.9} that
\begin{eqnarray}\label{4.11}
  &&\int_{t_{1}}^{t_{2}}\|u_{t}\|_{2}^{2} \textrm{d}t +\int_{\Omega} \left(u_{t}(t_{1})u(t_{1})-u_{t}(t_{2})u(t_{2}) \right) \textrm{d}x \nonumber\\
  &\leq& \int_{t_{1}}^{t_{2}}\|u_{t}\|_{2}^{2} \textrm{d}t +\|u_{t}(t_{1})\|_{2}\|u(t_{1})\|_{2}+\|u_{t}(t_{2})\|_{2}\|u(t_{2})\|_{2}\nonumber\\
  &\leq&B_{2}D^{2}(t)+8B_{2}\left( \frac{2q}{\mu_{0} a_{0}(q-2)}  \right)^{\frac{1}{2}}\sup_{t_{1}\leq s\leq t_{2}}E^{\frac{1}{2}}(s) D^{2}(t).
\end{eqnarray}
For the last term of \eqref{4.10}, we have
\begin{eqnarray}\label{4.12}
\int_{t_{1}}^{t_{2}}\int_{\Omega} g(u_{t})u \textrm{d}x\textrm{d}t &=&
\int_{t_{1}}^{t_{2}}\int_{|u_{t}|\leq 1} g(u_{t})u \textrm{d}x\textrm{d}t+
\int_{t_{1}}^{t_{2}}\int_{|u_{t}|>1} g(u_{t})u \textrm{d}x\textrm{d}t .\nonumber\\
\end{eqnarray}
On the one hand, by \eqref{3.30}, \eqref{04.7} and H\"older's inequality with $\frac{1}{p+2}+\frac{p+1}{p+2}=1$, we obtain
\begin{eqnarray}\label{4.13}
\int_{t_{1}}^{t_{2}}\int_{|u_{t}|>1} g(u_{t})u \textrm{d}x\textrm{d}t
&\leq&
c_{2}\int_{t_{1}}^{t_{2}}\int_{|u_{t}|>1} |u_{t}|^{p+1}u \textrm{d}x\textrm{d}t \nonumber\\
&\leq& c_{2}\int_{t_{1}}^{t_{2}} \| u \|_{p+2} \left(\int_{|u_{t}|>1} |u_{t}|^{p+2} dx\right) ^{\frac{p+1}{p+2}} \textrm{d}t \nonumber\\
&\leq &  c_{2} B_{3} \int_{t_{1}}^{t_{2}} \| \nabla u \|_{2} \left(\int_{|u_{t}|>1} |u_{t}|^{p+2} dx\right) ^{\frac{p+1}{p+2}} \textrm{d}t \nonumber\\
&\leq & c_{2} B_{3} \left(\frac{1}{c_{1}}\right)^{\frac{p+1}{p+2}}\left( \frac{2q}{\mu_{0} a_{0}(q-2)}  \right)^{\frac{1}{2}} \sup_{t_{1}\leq s\leq t_{2}}E^{\frac{1}{2}}(s) D^{p+1}(t),\nonumber\\
\end{eqnarray}
where $B_{3}$ is the best constant of Sobolev embedding $H_{0}^{1}(\Omega)\hookrightarrow L^{p+2}(\Omega)$. On the other hand, it follows from \eqref{04.7} that
\begin{eqnarray}\label{4.27}
&&\int_{t_{1}}^{t_{2}}\int_{|u_{t}|\leq 1} g(u_{t})u \textrm{d}x\textrm{d}t \nonumber\\
 &\leq&
\int_{t_{1}}^{t_{2}}  \|u\|_{2}\int_{|u_{t}|\leq 1} |g(u_{t})|^{2} \textrm{d}x \textrm{d}t \nonumber\\
&\leq& \int_{t_{1}}^{t_{2}}  \|u\|_{2}\int_{|u_{t}|\leq 1} |u_{t}g(u_{t})|^{\frac{2}{p+2}} \textrm{d}x \textrm{d}t \nonumber\\
&\leq& C_{3}\int_{t_{1}}^{t_{2}}  \|u\|_{2} \left(\int_{|u_{t}|\leq 1} |u_{t}g(u_{t})| \textrm{d}x\right)^{\frac{2}{p+2}} \textrm{d}t \nonumber\\
&\leq&  C_{3}B_{4} \left( \frac{2q}{\mu_{0} a_{0}(q-2)}  \right)^{\frac{1}{2}} \sup_{t_{1}\leq s\leq t_{2}}E^{\frac{1}{2}}(s) \left(\int_{t_{1}}^{t_{2}}  \int_{|u_{t}|\leq 1} |u_{t}g(u_{t})| \textrm{d}x \textrm{d}t \right)^{\frac{2}{p+2}}\nonumber\\
&\leq&C_{3}B_{4} \left( \frac{2q}{\mu_{0} a_{0}(q-2)}  \right)^{\frac{1}{2}} \sup_{t_{1}\leq s\leq t_{2}}E^{\frac{1}{2}}(s) D^{2}(t),
\end{eqnarray}
where $  C_{3}$ and $B_{4}$ are appropriate positive constants.

Substituting \eqref{4.11}--\eqref{4.27} into \eqref{4.10}, we obtain
\begin{eqnarray}\label{4.28}
\int_{t_{1}}^{t_{2}} I(u) \textrm{d}t
&\leq&B_{2}D^{2}(t)+(8B_{2}+C_{3}B_{4})\left( \frac{2q}{\mu_{0} a_{0}(q-2)}  \right)^{\frac{1}{2}}\sup_{t_{1}\leq s\leq t_{2}}E^{\frac{1}{2}}(s) D^{2}(t)\nonumber\\
&&+c_{2} B_{3} \left(\frac{1}{c_{1}}\right)^{\frac{p+1}{p+2}} \left( \frac{2q}{\mu_{0} a_{0}(q-2)}  \right)^{\frac{1}{2}} \sup_{t_{1}\leq s\leq t_{2}}E^{\frac{1}{2}}(s) D^{p+1}(t)\nonumber\\
&:=& B_{2}D^{2}(t)+C_{4}\sup_{t_{1}\leq s\leq t_{2}}E^{\frac{1}{2}}(s) D^{2}(t)+C_{5}\sup_{t_{1}\leq s\leq t_{2}}E^{\frac{1}{2}}(s) D^{p+1}(t).\nonumber\\
\end{eqnarray}
By \eqref{4.3} and the definitions of $E$ and $I$, we have
\begin{equation*}\label{4.16}
  \|u\|_{q}^{q}\leq B_{5} \| \nabla u\|_{2} \leq \frac{B_{5}}{\mu_{0}a_{0}}\mu(t)a(u,u) \leq \frac{B_{5}}{\mu_{0}a_{0} \theta} I(u),
\end{equation*}
where $B_{5}$ is also a embedding constant. Hence, we can deduce that
\begin{equation}\label{4.17}
  E(t)\leq \frac{1}{2} \|u_{t}\|_{2}^{2} + C_{6}I(u),
\end{equation}
where $C_{6}=\frac{1}{q}+\frac{q-2}{2q\theta}+\frac{B_{5}}{q^{2}\mu_{0}a_{0} \theta}$.

Now, we integrate \eqref{4.17} over $(t_{1},t_{2})$,
\begin{equation}\label{4.18}
  \int_{t_{1}}^{t_{2}}E(t) \textrm{d}t \leq \frac{1}{2}  \int_{t_{1}}^{t_{2}} \|u_{t}\|_{2}^{2} \textrm{d}t + C_{6} \int_{t_{1}}^{t_{2}} I(u) \textrm{d}t.
\end{equation}
It follows from integrating \eqref{2.7} over $(t,t_{2})$ that
\begin{equation*}
   E(t)= E(t_{2}) +\int_{t}^{t_{2}}\int_{\Omega} g(u_{t}) u_{t}  \textrm{d}x\textrm{d}t - \frac{1}{2} \int_{t}^{t_{2}} \mu'(t) a(u,u) \textrm{d}t .
\end{equation*}
Since $t_{2}-t_{1} \geq \frac{1}{2}$, then
\begin{equation*}
  E(t_{2}) \leq 2\int_{t_{1}}^{t_{2}} E(t) \textrm{d}t .
\end{equation*}
By \eqref{4.6}, we have
\begin{equation}\label{4.19}
  E(t)=E(t+1) +D^{p+2}(t) \leq E(t_{2})+D^{p+2}(t)
  \leq 2\int_{t_{1}}^{t_{2}} E(t) \textrm{d}t +D^{p+2}(t).
\end{equation}
Combining \eqref{4.8}, \eqref{4.28}, \eqref{4.18} and \eqref{4.19}, we have

\begin{eqnarray*}\label{04.10}
E(t) &\leq& (B_{2}+2C_{6}B_{2})D^{2}(t)+D^{p+2}(t) +2C_{4}C_{6}\sup_{t_{1}\leq s\leq t_{2}}E^{\frac{1}{2}}(s) D^{2}(t)\nonumber\\
&&+ 2C_{5}C_{6} \sup_{t_{1}\leq s\leq t_{2}}E^{\frac{1}{2}}(s) D^{p+1}(t)\nonumber\\
&\leq&(B_{2}+2C_{6}B_{2})D^{2}(t) +2C_{6}(C_{4} D^{2}(t)+C_{5}  D^{p+1}(t))E^{\frac{1}{2}}(t)
+D^{p+2}(t).\nonumber\\
\end{eqnarray*}
Using Young's inequality, we obtain
\begin{equation}\label{04.11}
  E(t)\leq C_{7}(D^{2}(t)+ D^{4}(t)+D^{p+2}(t)+D^{p+3}(t)+D^{2(p+2)}(t)  ),
\end{equation}
where $C_{7}>1$ is a positive constant. It follows from \eqref{4.6} that $E(0)\geq E(0)\geq D^{p+2}$. Then, we have
\begin{eqnarray*}\label{04.27}
  E(t)&\leq& C_{7}\Big( 1+ D^{2}(t)+D^{p }(t)+D^{p+1}(t)+D^{2(p+1)}(t) \Big)D^{2}(t)\nonumber\\
  &\leq& C_{7}\Big( 1+E(0)^{\frac{2}{p+2}}+E(0)^{\frac{p}{p+2}}+E(0)^{\frac{p+1}{p+2}}
  +E(0)^{\frac{2p+2}{p+2}}  \Big)D^{2}(t),\nonumber\\
\end{eqnarray*}
which implies that
\begin{equation*}\label{4.29}
  E(t)^{\frac{p+2}{2}} \leq C_{8}^{\frac{p+2}{2}}D^{p+2}(t)= C_{8}^{\frac{p+2}{2}} (E(t)-E(t+1)),
\end{equation*}
where $C_{8}=C_{7}\Big( 1+E(0)^{\frac{2}{p+2}}+E(0)^{\frac{p}{p+2}}+E(0)^{\frac{p+1}{p+2}}
  +E(0)^{\frac{2p+2}{p+2}}  \Big)>1$.
Hence, by Lemma \ref{lemma4.1}, we have
\begin{eqnarray*}
  &E(t)\leq \Big(E(0)^{-\frac{p}{2}}
  +k_{1}^{-1}\frac{p}{2}[t-1]^{+}\Big)^{-\frac{2}{p+2}}, &\mbox{ if } p>0;\\
  &E(t)\leq E(0)e^{-k_{2}[t-1]^{+}}, &\mbox{ if } p=0,
\end{eqnarray*}
where $k_{1}=C_{8}^{\frac{p+2}{4}}>1$, $k_{2}=\log\left( \frac{k_{1}}{k_{1}-1}  \right)$ and $E(0)>0$.
The proof of Theorem \ref{theorem4.2} is complete.
$\hfill\qedsymbol$

\section{Blow up}


\subsection{Proof of Theorem \ref{theorem5.1}}

In order to establish the blow-up result, we need to restate the properties of the energy functional $E(t)$. By $({\bf{H3'}})$ and Lemma \ref{lemma2.1}, we have
\begin{equation}\label{5.1}
  \frac{d}{\textrm{d}t}E(t)=\frac{1}{2}\mu'(t)a(u,u)-\|u_{t}\|_{p+2}^{p+2}<0.
\end{equation}

\noindent\textbf{\emph{Proof of Theorem \ref{theorem5.1}.}}
If the conclusion is false, then the weak solution $u$ can be extended to the whole interval $[0,+\infty)$. By $E(0)<0$ and \eqref{5.1}, we have $E(t)<0$ on $[0,+\infty)$. Here, we define $G(t)=-E(t)$, then we have $G(t)\geq G(0)>0$ for all $t \geq 0$.

We define a auxiliary function
\begin{equation}\label{5.2}
  Y(t)=G^{1-\alpha}(t)+\varepsilon \int_{\Omega}u_{t}u \textrm{d}x,
\end{equation}
where $\varepsilon$ is a positive constant which will be determined later,
and $\alpha$ satisfies
\begin{equation}\label{5.15}
0<\alpha=\frac{q-p-2}{q(p+1)} < \frac{q-2}{2q}\leq \frac{1}{2}.
\end{equation}
Taking the first derivative of the function $Y(t)$, we have
\begin{eqnarray}\label{5.3}
Y'(t)&=&(1-\alpha)G^{-\alpha}(t)G'(t)+\varepsilon \|u_{t}\|_{2}^{2} +\varepsilon\int_{\Omega}u u_{tt} \textrm{d}x \nonumber\\
&=&(1-\alpha)G^{-\alpha}(t)G'(t)+\varepsilon \|u_{t}\|_{2}^{2} -\varepsilon \mu(t) a(u,u) +\varepsilon\int_{\Omega} |u|^{q}\log |u| \textrm{d}x \nonumber\\
&&-\varepsilon\int_{\Omega} |u_{t}|^{p} u_{t} u \textrm{d}x.\nonumber\\
\end{eqnarray}
Next, we estimate the last term on the right hand of \eqref{5.3}. Using H\"older's inequality and Young's inequality with $\frac{1}{p+2}+\frac{p+1}{p+2}=1$, we have
\begin{eqnarray}\label{5.4}
  \varepsilon\int_{\Omega} |u_{t}|^{p} u_{t} u \textrm{d}x\leq \varepsilon \|u\|_{p+2}\|u_{t}\|_{p+2}^{p+1}\leq \varepsilon \left( \frac{\epsilon^{p+2}}{p+2}\|u\|_{p+2}^{p+2}+\frac{p+1}{p+2} \epsilon^{-\frac{p+2}{p+1}}\|u_{t}\|_{p+2}^{p+2} \right),\nonumber\\
\end{eqnarray}
where $\epsilon$ is a positive constant which will be determined later. Substituting \eqref{5.1} and \eqref{5.4} into \eqref{5.3}, we obtain
\begin{eqnarray}\label{5.5}
Y'(t)&\geq&(1-\alpha)G^{-\alpha}(t)G'(t)+\varepsilon \|u_{t}\|_{2}^{2} -\varepsilon \mu(t) a(u,u) +\varepsilon\int_{\Omega} |u|^{q}\log |u| \textrm{d}x \nonumber\\ &&-\varepsilon \left( \frac{\epsilon^{p+2}}{p+2}\|u\|_{p+2}^{p+2}+\frac{p+1}{p+2} \epsilon^{-\frac{p+2}{p+1}}\|u_{t}\|_{p+2}^{p+2} \right)\nonumber\\
&=&\left( (1-\alpha)G^{-\alpha}(t)-\varepsilon\frac{p+1}{p+2} \epsilon^{-\frac{p+2}{p+1}}\right) G'(t)+\varepsilon \|u_{t}\|_{2}^{2}  +\varepsilon\int_{\Omega} |u|^{q}\log |u| \textrm{d}x\nonumber\\
&&+\varepsilon \left( -\frac{1}{2}\mu'(t)\frac{p+1}{p+2} \epsilon^{-\frac{p+2}{p+1}}-   \mu(t)\right) a(u,u) -\varepsilon  \frac{\epsilon^{p+2}}{p+2}\|u\|_{p+2}^{p+2}.
\end{eqnarray}
Let $\epsilon^{-\frac{p+2}{p+1}}=\eta G^{-\alpha}(t)>0$. Then, substituting \eqref{2.6} into \eqref{5.5}, we have
\begin{eqnarray}\label{5.6}
Y'(t)&\geq&\left( (1-\alpha)-\varepsilon \eta\frac{p+1}{p+2} \right)G^{-\alpha}(t)G'(t)+\varepsilon(\frac{q}{2}+1)\|u_{t}\|_{2}^{2}\nonumber\\
&&+\varepsilon \left( -\frac{1}{2}\mu'(t)\frac{p+1}{p+2} \eta G^{-\alpha}(t) +\frac{q-2}{2}\mu(t)\right) a(u,u)\nonumber\\
&&+\frac{\varepsilon}{q}\|u\|_{q}^{q}+\varepsilon qG(t) -\frac{\varepsilon}{p+2}\eta^{-(p+1)}G^{\alpha(p+1)}(t) \|u\|_{p+2}^{p+2}.
\end{eqnarray}
Here, we estimate the last term of \eqref{5.6}. It follows from Young's inequality with $\frac{q-p-2}{q}+\frac{p+2}{q}=1$ and the condition $q>p+2$ that
\begin{equation}\label{5.7}
  G^{\alpha(p+1)}(t) \|u\|_{p+2}^{p+2} \leq \frac{q-p-2}{q}G^{\frac{\alpha q(p+1)}{q-p-2}}(t)+\frac{p+2}{q}\|u\|_{p+2}^{q}.
\end{equation}
By taking $\alpha=\frac{q-p-2}{q(p+1)}$ and substituting \eqref{5.7} into \eqref{5.6}, we obtain
\begin{eqnarray*}\label{5.8}
Y'(t)&\geq& \left( (1-\alpha)-\varepsilon \eta\frac{p+1}{p+2} \right)G^{-\alpha}(t)G'(t)+\varepsilon(\frac{q}{2}+1)\|u_{t}\|_{2}^{2}\nonumber\\
&&+\varepsilon \left( -\frac{1}{2}\mu'(t)\frac{p+1}{p+2} \eta G^{-\alpha}(t) +\frac{q-2}{2}\mu(t)\right) a(u,u) +\frac{\varepsilon}{q}\|u\|_{q}^{q}\nonumber\\
&&+\varepsilon qG(t) -\frac{\varepsilon}{p+2}\eta^{-(p+1)}\left( \frac{q-p-2}{q}G^{\frac{\alpha q(p+1)}{q-p-2}}(t)+\frac{p+2}{q}\|u\|_{p+2}^{q} \right)\nonumber\\
&=& \left( (1-\alpha)-\varepsilon \eta\frac{p+1}{p+2} \right)G^{-\alpha}(t)G'(t)+\varepsilon(\frac{q}{2}+1)\|u_{t}\|_{2}^{2}\nonumber\\
&&+\varepsilon \left( -\frac{1}{2}\mu'(t)\frac{p+1}{p+2} \eta G^{-\alpha}(t) +\frac{q-2}{2}\mu(t)\right) a(u,u) \nonumber\\
&&+\frac{\varepsilon}{q} \left(\|u\|_{q}^{q}- \eta^{-(p+1)}\|u\|_{p+2}^{q}\right)+\varepsilon \left( q -\frac{1}{p+2}\eta^{-(p+1)} \frac{q-p-2}{q} \right)G(t)\nonumber\\
&\geq& \left( (1-\alpha)-\varepsilon \eta\frac{p+1}{p+2} \right)G^{-\alpha}(t)G'(t)+\varepsilon(\frac{q}{2}+1)\|u_{t}\|_{2}^{2}\nonumber\\
&&+\varepsilon \left( -\frac{1}{2}\mu'(t)\frac{p+1}{p+2} \eta G^{-\alpha}(t) +\frac{q-2}{2}\mu(t)\right) a(u,u) \nonumber\\
&&+\frac{\varepsilon}{q} \left(1- \eta^{-(p+1)}B_{6}\right) \|u\|_{q}^{q}+\varepsilon \left( q -\frac{1}{p+2}\eta^{-(p+1)} \frac{q-p-2}{q} \right)G(t),\nonumber\\
\end{eqnarray*}
where $B_{6}$ is a embedding constant.

Let's make $\eta$  sufficiently large and $\varepsilon$ sufficiently small such that
\begin{eqnarray*}
&&1- \eta^{-(p+1)}B_{6} > 0,\\
&&q -\frac{1}{p+2}\eta^{-(p+1)} \frac{q-p-2}{q} > 0 ,\\
&&(1-\alpha)-\varepsilon \eta\frac{p+1}{p+2} > 0,\\
&&Y(t)=G^{1-\alpha}(t)+\varepsilon \int_{\Omega}u_{t}u \textrm{d}x >0.
\end{eqnarray*}
Moreover, due to ({\bf{H2}}) and $q>p+2$, we have
\begin{equation*}
  -\frac{1}{2}\mu'(t)\frac{p+1}{p+2} \eta G^{-\alpha}(t) +\frac{q-2}{2}\mu(t) >0.
\end{equation*}
Then, we can define
\begin{equation}\label{5.9}
  \xi_{1}=\varepsilon \min\{ \frac{q}{2}+1, \frac{q-2}{2}\mu_{0}, 1-\eta^{-(p+1)}B_{6}, q -\frac{1}{p+2}\eta^{-(p+1)} \frac{q-p-2}{q}  \},
\end{equation}
to conclude
\begin{equation*}\label{5.10}
  Y'(t)\geq  \xi_{1}\left( G(t)+\|u_{t}\|_{2}^{2}+a(u,u)+\|u\|_{q}^{q} \right)\geq 0, \mbox{ for all } t\geq 0.
\end{equation*}
Finally, we prove
\begin{equation}\label{5.11}
   Y'(t)\geq  \xi_{1}\left( G(t)+\|u_{t}\|_{2}^{2}+a(u,u)+\|u\|_{q}^{q} \right)\geq \xi_{2}Y^{\frac{1}{1-\alpha}}(t), \mbox{ for all } t\geq 0.
\end{equation}

\textbf{case 1.} If $\int_{\Omega}u u_{t} \textrm{d}x \leq 0$ for some $t\geq 0$, it is easy to get
\begin{equation*}
  Y^{\frac{1}{1-\alpha}}(t)=\left( G^{1-\alpha}(t)+\varepsilon \int_{\Omega}u u_{t} \textrm{d}x \right)^{\frac{1}{1-\alpha}}\leq G(t).
\end{equation*}
In this case, we have
\begin{equation*}
  Y'(t)\geq  \xi_{1} G(t)\geq  \xi_{1}Y^{\frac{1}{1-\alpha}}(t).
\end{equation*}
Thus, \eqref{5.11} holds for all $t\geq 0$ such that $\int_{\Omega}u u_{t} \textrm{d}x \leq 0$.

\textbf{case 2.} If $\int_{\Omega}u u_{t} \textrm{d}x > 0$ for some $t\geq 0$, then by $\alpha \in (0,\frac{1}{2})$, we have
\begin{eqnarray}\label{5.12}
  Y^{\frac{1}{1-\alpha}}(t)&=&\left( G^{1-\alpha}(t)+\varepsilon \int_{\Omega}u u_{t} \textrm{d}x \right)^{\frac{1}{1-\alpha}} \leq 2^{\frac{1}{1-\alpha}}\left( G(t)+\varepsilon^{\frac{1}{1-\alpha}}(\int_{\Omega} uu_{t} \textrm{d}x)^{\frac{1}{1-\alpha}}  \right).\nonumber\\
\end{eqnarray}
It follows from H\"older's inequality and Young's inequality with $\frac{1}{\theta_{1}}+\frac{1}{\theta_{2}}=1$ that
\begin{equation*}\label{5.13}
  (\int_{\Omega} uu_{t} \textrm{d}x)^{\frac{1}{1-\alpha}} \leq \|u\|_{2}^{\frac{1}{1-\alpha}} \|u_{t}\|_{2}^{\frac{1}{1-\alpha}} \leq C\left( \|u\|_{2}^{\frac{\theta_{1}}{1-\alpha}} +\|u_{t}\|_{2}^{\frac{\theta_{2}}{1-\alpha}} \right).
\end{equation*}
By taking $\theta_{2}=2(1-\alpha)>1, \theta_{1}=\frac{1-\alpha}{1-2\alpha}$ and Sobolev inequality, we have
\begin{equation}\label{5.14}
  (\int_{\Omega} uu_{t} \textrm{d}x)^{\frac{1}{1-\alpha}} \leq
  C\left( \|u\|_{2}^{\frac{2}{1-2\alpha}} +\|u_{t}\|_{2}^{2} \right) \leq C \left( \big(\|u\|_{q}^{q} \big)^{\frac{2}{q(1-2\alpha)}} +\|u_{t}\|_{2}^{2} \right)
\end{equation}
From \eqref{5.15}, we know $0<\frac{2}{q(1-2\alpha)}<1$. Since
\begin{equation}\label{5.16}
  \chi^{\kappa}\leq \chi+1 \leq (1+\frac{1}{\alpha_{1}})(\chi+\alpha_{1}), \mbox{ for all } \chi \geq 0, 0<\kappa\leq 1, \alpha_{1} >0,
\end{equation}
then replacing $\chi$ with $\|u\|_{q}^{q}$ and $\kappa$ with $\frac{2}{q(1-2\alpha)}$ in \eqref{5.16}, we obtain
\begin{equation}\label{5.17}
  \left(\|u\|_{q}^{q} \right)^{\frac{2}{q(1-2\alpha)}}\leq d\left( \|u\|_{q}^{q} +G(t_{1})\right)\leq d \left( \|u\|_{q}^{q}+G(t)\right),
\end{equation}
where $d=1+\frac{1}{G(t_{1})}$. In this case, \eqref{5.12}, \eqref{5.14} and \eqref{5.17} yield
\begin{eqnarray*}\label{5.18}
Y^{\frac{1}{1-\alpha}}
&\leq& 2^{\frac{1}{1-\alpha}}\left( G(t)+\varepsilon^{\frac{1}{1-\alpha}}(d\|u\|_{q}^{q} +dG(t)+\|u_{t}\|_{2}^{2})\right)\nonumber\\
&\leq&\xi_{3}\left( G(t)+\|u\|_{q}^{q}+\|u_{t}\|_{2}^{2} \right),
\end{eqnarray*}
where
\begin{equation*}
  \xi_{3}= 2^{\frac{1}{1-\alpha}}(1+\varepsilon^{\frac{1}{1-\alpha}}d).
\end{equation*}
Let $\xi=\frac{ \xi_{1}}{\xi_{3}}$, then
\begin{equation*}
  Y'(t)\geq \xi Y^{\frac{1}{1-\alpha}}(t).
\end{equation*}
Thus, \eqref{5.11} holds for all $t\geq 0$ such that $\int_{\Omega} u u_{t}>0$.

Obviously, by \textbf{case 1} and \textbf{2},  we can draw the conclusion that for all $t\geq 0$, we have
\begin{equation}\label{5.19}
  Y'(t)\geq \xi Y^{\frac{1}{1-\alpha}}(t), \ \ \mbox{for} \ \ \xi>0 \ \  \mbox{and} \ \ 0<\frac{1}{1-\alpha}<2.
\end{equation}
Hence, a simple calculation of \eqref{5.19} yields that $Y(t)$ blows up in finite time and the lifespan of solution to problem \eqref{1.1} is finite which contradicts with the assumption.
$\hfill\qedsymbol$

\subsection{Proof of Theorem \ref{theorem5.2}}

Before prove Theorem \ref{theorem5.2}, we need the following two lemmas which are crucial in the proof of this theorem. Let us start by introducing the following functional on $H_{0}^{1}(\Omega)$:
\begin{equation*}
  F(t)=\frac{1}{q}\|u(t)\|_{q}^{q}+I(u(t)).
\end{equation*}

\begin{lem}\label{lemma5.1}
If $F(t)<0$ for all $t>0$, we have that $a(u(t),u(t))>r_{*}^{2}$ for all $t>0$.
\end{lem}
\begin{proof}
We know that $I(u(t))<0$ from the condition $F(t)<0$. Then by Corollary \ref{corollary2.1}, we have that $a(u(t),u(t))>r_{*}^{2}$ for all $t>0$.
\end{proof}

\begin{lem}\label{lemma5.2}
If $a(u_{0},u_{0})>r_{*}^{2}$ and $0<E(0)<M$, we have that $F(t)<0$ for all $t>0$.
\end{lem}
\begin{proof}
Firstly, we prove that $F(0)<0$. If it is false, then $F(0)\geq 0$. We can obtain that
\begin{equation*}
  E(0)\geq J(0) \geq \frac{1}{q^{2}}\|u_{0}\|_{q}^{q}+\frac{q-2}{2q}\mu_{0}a(u_{0},u_{0})
  >\frac{q-2}{2q}\mu_{0}r_{*}^{2}=M,
\end{equation*}
which is wrong.

Next, we assume that the conclusion of this lemma is false. Then, by continuity, there exists $t_{0}$ such that $F(t)<0$ over $[0,t_{0})$ and $F(t_{0})=0$. By Lemma \ref{lemma5.1}, we have $a(u(t),u(t))>r_{*}^{2}$ for all $0\leq t<t_{0}$. And from the continuity of $t\mapsto a(u(t),u(t))$, we have $a(u(t_{0}),u(t_{0}))\geq r_{*}^{2}$. Thus,
\begin{equation*}
  E(t_{0})=\frac{1}{2}\|u_{t}(t_{0})\|_{2}^{2}+J(u(t_{0}))
  >\frac{q-2}{2q}\mu_{0}a(u_{0},u_{0})+\frac{1}{q}F(t_{0})
  \geq \frac{q-2}{2q}\mu_{0}r_{*}^{2}=M,
\end{equation*}
which is a contradiction.
\end{proof}

\begin{rem}\label{remark5.1}
By the Lemma \ref{lemma5.1} and Lemma \ref{lemma5.2}, we can deduce that if $a(u_{0},u_{0})>r_{*}^{2}$ and $0<E(0)<M$, then $a(u(t),u(t))>r_{*}^{2}$ for all $t>0$ and
\begin{equation*}
  F(t)=\frac{1}{q}\|u(t)\|_{q}^{q}+I(u(t))<0, ~\text{for all}~t>0,
\end{equation*}
which implies that $I(u(t))<0$ for all $t>0$.
\end{rem}

\noindent\textbf{\emph{Proof of Theorem \ref{theorem5.2}.}}
If the conclusion is false, then the weak solution $u$ can be extended to the whole interval $[0,+\infty)$.

By taking Cauchy inequality with a positive constant $\varepsilon_{1} \in (0,1)$  and \eqref{2.6}, we have
\begin{eqnarray}\label{5.20}
&&\frac{\textrm{d}}{\textrm{d}t}\int_{\Omega} u u_{t} \textrm{d}x \nonumber\\ &=&\|u_{t}\|_{2}^{2} +\int_{\Omega} u u_{tt} \textrm{d}x \nonumber\\
&=&\|u_{t}\|_{2}^{2}-\mu(t)a(u,u)-\int_{\Omega}|u_{t}|^{p}u_{t}u \textrm{d}x +\int_{\Omega}|u|^{q}\log|u| \textrm{d}x \nonumber\\
&=& \left( 1+\frac{q(1-\varepsilon_{1}) }{2} \right) \|u_{t}\|_{2}^{2} +\left( \frac{q(1-\varepsilon_{1}) }{2}-1\right) \mu(t)a(u,u) \nonumber\\
&&+\varepsilon_{1}\int_{\Omega}|u|^{q}\log|u| \textrm{d}x +\frac{1-\varepsilon_{1}}{q} \|u\|_{q}^{q} -q(1-\varepsilon_{1}) E(t)- \int_{\Omega}|u_{t}|^{p}u_{t}u \textrm{d}x.\nonumber\\
\end{eqnarray}
Now, we estimate the last term on the right hand of \eqref{5.20}. It follows from \eqref{5.4} that
\begin{equation}\label{5.21}
  \int_{\Omega} |u_{t}|^{p} u_{t} u \textrm{d}x\leq  \|u\|_{p+2}\|u_{t}\|_{p+2}^{p+1}\leq  \frac{\epsilon^{p+2}}{p+2}\|u\|_{p+2}^{p+2}+\frac{p+1}{p+2} \epsilon^{-\frac{p+2}{p+1}}\|u_{t}\|_{p+2}^{p+2} .
\end{equation}
By the fact that $q>p+2$ and $\frac{\chi^{m}}{m}$ is a convexity function for all $\chi \geq 0$ and $m>0$, we obtain
\begin{equation}\label{5.22}
  \frac{\|u\|_{p+2}^{p+2}}{p+2} \leq \theta \frac{\|u\|_{2}^{2}}{2} +(1-\theta)\frac{\|u\|_{q}^{q}}{q},
\end{equation}
where $\theta$ is a positive constant in $(0,1)$. Combining \eqref{5.21} and \eqref{5.22}, we have
\begin{equation}\label{5.23}
  \int_{\Omega} |u_{t}|^{p} u_{t} u \textrm{d}x\leq
  \epsilon^{p+2}\left( \theta \frac{\|u\|_{2}^{2}}{2} +(1-\theta)\frac{\|u\|_{q}^{q}}{q} \right)+\frac{p+1}{p+2} \epsilon^{-\frac{p+2}{p+1}}\|u_{t}\|_{p+2}^{p+2} .
\end{equation}
Substituting \eqref{5.23} into \eqref{5.20}, then
\begin{eqnarray}\label{5.24}
\frac{\textrm{d}}{\textrm{d}t}\int_{\Omega} u u_{t} \textrm{d}x
&\geq& \left( 1+\frac{q(1-\varepsilon_{1}) }{2} \right) \|u_{t}\|_{2}^{2} +\left( \frac{q(1-\varepsilon_{1}) }{2}-1\right) \mu(t)a(u,u) \nonumber\\
&&+\varepsilon_{1}\int_{\Omega}|u|^{q}\log|u| \textrm{d}x +\frac{1-\varepsilon_{1}}{q} \|u\|_{q}^{q} -q(1-\varepsilon_{1}) E(t)\nonumber\\
&&-  \epsilon^{p+2}\left( \theta \frac{\|u\|_{2}^{2}}{2} +(1-\theta)\frac{\|u\|_{q}^{q}}{q} \right)-\frac{p+1}{p+2} \epsilon^{-\frac{p+2}{p+1}}\|u_{t}\|_{p+2}^{p+2}. \nonumber\\
\end{eqnarray}
To handle the last term of \eqref{5.24},  we subtract $\frac{p+1}{p+2} \epsilon^{-\frac{p+2}{p+1}}\frac{d}{\textrm{d}t} E(t)$ from both side of above inequality. And by Remark \ref{remark5.1}, we have $I(u(t))< 0$, for all $t>0$. Then
\begin{eqnarray}\label{5.25}
&&\frac{\textrm{d}}{\textrm{d}t}\left(\int_{\Omega} u u_{t} \textrm{d}x  -\frac{p+1}{p+2} \epsilon^{-\frac{p+2}{p+1}}E(t) \right)\nonumber\\
&=&\frac{\textrm{d}}{\textrm{d}t}\int_{\Omega} u u_{t} \textrm{d}x - \frac{p+1}{p+2} \epsilon^{-\frac{p+2}{p+1}}\left( \frac{1}{2}\mu'(t)a(u,u)-\int_{\Omega} |u_{t}|^{p+2} \textrm{d}x \right)\nonumber\\
&=&\left( 1+\frac{q(1-\varepsilon_{1}) }{2} \right) \|u_{t}\|_{2}^{2} +\left( \frac{q(1-\varepsilon_{1}) }{2}-1\right) \mu(t)a(u,u) \nonumber\\
&&+\varepsilon_{1}\int_{\Omega}|u|^{q}\log|u| \textrm{d}x +\frac{1}{q}\left( 1-\varepsilon_{1} -\epsilon^{p+2}(1-\theta) \right) \|u\|_{q}^{q} -q(1-\varepsilon_{1}) E(t)\nonumber\\
&&-  \epsilon^{p+2}\theta \frac{\|u\|_{2}^{2}}{2} -\frac{p+1}{p+2} \epsilon^{-\frac{p+2}{2(p+1)}}\mu'(t)a(u,u) \nonumber\\
&\geq&\left( 1+\frac{q(1-\varepsilon_{1}) }{2} \right) \|u_{t}\|_{2}^{2} +\left( \frac{q(1-\varepsilon_{1}) }{2}-1\right) \mu(t)a(u,u) \nonumber\\
&&+\varepsilon_{1}\mu(t)a(u,u)-  \epsilon^{p+2}\theta \frac{\|u\|_{2}^{2}}{2} +\frac{1}{q}\left( 1-\varepsilon_{1} -\epsilon^{p+2}(1-\theta) \right) \|u\|_{q}^{q} \nonumber\\
&& -q(1-\varepsilon_{1}) E(t)-\frac{p+1}{p+2} \epsilon^{-\frac{p+2}{2(p+1)}}\mu'(t)a(u,u) \nonumber\\
&\geq&\left( 1+\frac{q(1-\varepsilon_{1}) }{2} \right) \|u_{t}\|_{2}^{2} +\left( \frac{q(1-\varepsilon_{1}) }{2}-1\right) \mu(t) a(u,u) \nonumber\\
&&+\left( \varepsilon_{1}a_{0}\mu_{0}-  \frac{\epsilon^{p+2}}{2}B_{7}^{2}\right) \|\nabla u\|_{2}^{2} +\frac{1}{q}\left( 1-\varepsilon_{1} -\epsilon^{p+2}(1-\theta) \right) \|u\|_{q}^{q} \nonumber\\
&& -q(1-\varepsilon_{1}) E(t)-\frac{p+1}{p+2} \epsilon^{-\frac{p+2}{2(p+1)}}\mu'(t)a(u,u),
\end{eqnarray}
where $B_{7}$ is the best constant of Sobolev embedding inequality $\|u\|_{2}^{2}\leq B_{7} \|\nabla u\|_{2}^{2}$.

Now, we take
\begin{equation*}\label{5.27}
  \epsilon=\left( \frac{2\varepsilon_{1}a_{0}\mu_{0}}{ B_{7}^{2}}  \right)^{\frac{1}{p+2}},
\end{equation*}
and
\begin{equation}\label{5.28}
  \frac{2\varepsilon_{1}a_{0}\mu_{0}}{2\varepsilon_{1}a_{0}\mu_{0}+B_{7}^{2}(1-\varepsilon_{1})} \leq \theta <1,
\end{equation}
such that the both parameters in front of $\| \nabla u\|_{2}^{2}$ and $\|u\|_{q}^{q}$ are non-negative, i.e.
\begin{eqnarray*}\label{5.26}
&&\varepsilon_{1}a_{0}\mu_{0}-  \frac{\epsilon^{p+2}}{2}B_{7}^{2}\geq 0, \nonumber\\
 &&1-\varepsilon_{1} -\epsilon^{p+2}(1-\theta)\geq 0.
\end{eqnarray*}
Then, \eqref{5.25} is reduced to
\begin{eqnarray}\label{5.29}
&&\frac{\textrm{d}}{\textrm{d}t}\left(\int_{\Omega} u u_{t} \textrm{d}x  -h_{1}(\varepsilon_{1})E(t) \right) \nonumber\\
&\geq& \left( 1+\frac{q(1-\varepsilon_{1}) }{2} \right) \|u_{t}\|_{2}^{2} +\left( \frac{q(1-\varepsilon_{1}) }{2}-1\right) a_{0}\mu_{0}\|\nabla u\|_{2}^{2}-q(1-\varepsilon_{1}) E(t) \nonumber\\
&\geq& 2\sqrt{\frac{a_{0}\mu_{0}}{B_{7}} \left( 1+\frac{q(1-\varepsilon_{1}) }{2} \right)\left( \frac{q(1-\varepsilon_{1}) }{2}-1\right)}\int_{\Omega} u u_{t} \textrm{d}x-q(1-\varepsilon_{1}) E(t) \nonumber\\
&=&h_{2}(\varepsilon_{1})\left( \int_{\Omega} u u_{t} \textrm{d}x-h_{3}(\varepsilon_{1}) E(t) \right),
\end{eqnarray}
where
\begin{eqnarray}\label{5.30}
&&h_{1}(\varepsilon_{1})=\frac{p+1}{p+2} \epsilon^{-\frac{p+2}{p+1}}=\frac{p+1}{p+2}\left(  \frac{ B_{7}^{2}}{2\varepsilon_{1}a_{0}\mu_{0}} \right)^{\frac{1}{p+1}}, \nonumber\\
&&h_{2}(\varepsilon_{1})=2\sqrt{\frac{a_{0}\mu_{0}}{B_{7}}\left( 1+\frac{q(1-\varepsilon_{1}) }{2} \right)\left( \frac{q(1-\varepsilon_{1}) }{2}-1\right)},\nonumber\\
&&h_{3}(\varepsilon_{1})=\frac{q(1-\varepsilon_{1})}{h_{2}(\varepsilon_{1})}.
\end{eqnarray}

By the definition of $h_{2}(\varepsilon_{1})$, there exists a constant $\varepsilon_{2}\in (0,1)$ such that
\begin{equation*}\label{5.31}
  h_{2}(\varepsilon_{2})=0 \ \ \mbox{and} \ \ h_{2}(\varepsilon_{1})>0 \ \ \mbox{for all } \ \ \varepsilon_{1}\in (0,\varepsilon_{2}).
\end{equation*}
Thus, we know that, as $\varepsilon_{1}\rightarrow 0^{+}$,
\begin{eqnarray}\label{5.32}
&&h_{1}(\varepsilon_{1})\rightarrow +\infty, \nonumber\\
&&h_{2}(\varepsilon_{1})\rightarrow \sqrt{\frac{a_{0}\mu_{0}(q^{2}-4)}{B_{7}}},\nonumber\\
&&h_{3}(\varepsilon_{1})\rightarrow \frac{q}{\sqrt{q^{2}-4}}.
\end{eqnarray}
And $\varepsilon_{1}\rightarrow \varepsilon_{2}^{-}$, we have
\begin{eqnarray}\label{5.33}
&&h_{1}(\varepsilon_{1})\rightarrow \frac{p+1}{p+2}\left(  \frac{ B_{7}^{2}}{2\varepsilon_{2}a_{0}\mu_{0}} \right)^{\frac{1}{p+1}}, \nonumber\\
&&h_{2}(\varepsilon_{1})\rightarrow 0,\nonumber\\
&&h_{3}(\varepsilon_{1})\rightarrow +\infty.
\end{eqnarray}
From \eqref{5.32}, \eqref{5.33} and the definitions of $h_{1}$ and $h_{3}$, there exists $\varepsilon'\in (0,\varepsilon_{2})\subset (0,1) $ such that
\begin{equation}\label{5.033}
h_{1}(\varepsilon')=h_{3}(\varepsilon').
\end{equation}
Then, \eqref{5.29} is reduced to
\begin{equation}\label{5.34}
  \frac{\textrm{d}}{\textrm{d}t}\left(\int_{\Omega} u u_{t} \textrm{d}x  -h_{1}(\varepsilon')E(t) \right)\geq h_{2}(\varepsilon')\left( \int_{\Omega} u u_{t} \textrm{d}x-h_{1}(\varepsilon') E(t) \right).
\end{equation}
Let $F(t)=\int_{\Omega} u u_{t} \textrm{d}x  -h_{1}(\varepsilon')E(t)$. By \eqref{05.19} and \eqref{5.34}, we have
\begin{equation*}
  F(0)=\int_{\Omega} u_{0}u_{1} \textrm{d}x  -h_{1}(\varepsilon')E(0)>0,
\end{equation*}
and
\begin{equation*}
  \frac{\textrm{d}}{\textrm{d}t}F(t)\geq h_{2}(\varepsilon')F(t),
\end{equation*}
which implies
\begin{equation}\label{5.35}
  F(t)\geq e^{h_{2}(\varepsilon')t}F(0), \ \ \mbox{for all } \ \ t\geq 0.
\end{equation}
By the fact that $E(t)\geq 0$ for all $t>0$, \eqref{5.35} is reduced to \begin{equation*}
  \int_{\Omega} u u_{t} \textrm{d}x \geq e^{h_{2}(\varepsilon')t}F(0), \ \ \mbox{for all } \ \ t\geq 0.
\end{equation*}
That is
\begin{eqnarray}\label{5.36}
\|u\|_{2}^{2}&=&\|u_{0}\|_{2}^{2}+2\int_{0}^{t}\int_{\Omega} u u_{\tau} dxd\tau \nonumber\\
&\geq&\|u_{0}\|_{2}^{2}+2\int_{0}^{t} e^{h_{2}(\varepsilon')\tau}F(0) d\tau \nonumber\\
&=&\|u_{0}\|_{2}^{2}+\frac{2}{h_{2}(\varepsilon')}\left( e^{h_{2}(\varepsilon')t}-1 \right)F(0),
\end{eqnarray}
which means that $\|u\|_{2}^{2}$ is growing exponentially as $t \rightarrow +\infty$.

Now, we derive the contradiction. By direct calculation, we have
\begin{eqnarray}\label{5.37}
\|u\|_{2}^{2}&=&\int_{\Omega}\left( u_{0}+\int_{0}^{t}u_{\tau} d\tau \right)^{2} \textrm{d}x \nonumber\\
&=&\|u_{0}\|_{2}^{2} +2\int_{\Omega}u_{0}\left(\int_{0}^{t}u_{\tau} d\tau \right) \textrm{d}x +\int_{\Omega}\left( \int_{0}^{t}u_{\tau} d\tau \right)^{2} \textrm{d}x \nonumber\\
&\leq& \|u_{0}\|_{2}^{2} +2\|u_{0}\|_{2} \left( \int_{\Omega}\left( \int_{0}^{t} u_{\tau} d\tau \right)^{2} \textrm{d}x \right)^{\frac{1}{2}} +\int_{\Omega}\left( \int_{0}^{t}u_{\tau} d\tau \right)^{2} \textrm{d}x \nonumber\\
&\leq& 2\|u_{0}\|_{2}^{2} + 2\int_{\Omega}\left( \int_{0}^{t}u_{\tau} d\tau \right)^{2} \textrm{d}x \nonumber\\
&\leq& 2\|u_{0}\|_{2}^{2} +2t\int_{\Omega} \int_{0}^{t} |u_{\tau}|^{2} d\tau \textrm{d}x \nonumber\\
&=&2\|u_{0}\|_{2}^{2} +2t\int_{0}^{t} \|u_{\tau}\|_{2}^{2} d\tau \nonumber\\
&\leq& 2\|u_{0}\|_{2}^{2} +2 tB_{8}\int_{0}^{t} \|u_{\tau}\|_{p+2}^{p+2} - \frac{1}{2}\mu'(\tau)a(u,u) d\tau \nonumber\\[8pt]
&=&2\|u_{0}\|_{2}^{2} +2 tB_{8}(E(0)-E(t))\nonumber\\[8pt]
&\leq&2\|u_{0}\|_{2}^{2} +2 tB_{8}E(0),
\end{eqnarray}
where $B_{8}$ is the embedding constant. \eqref{5.37} contradicts with \eqref{5.36}. Then, the weak solution of \eqref{1.1} blows up in finite time. $\hfill\qedsymbol$

\vskip 5mm

{\bf Ethics statement.} This work does not involve any active collection of human data.

\vskip 5mm

{\bf Data accessibility statement.} This work does not have any experimental data.

\vskip 5mm

{\bf Competing interests statement.} We have no competing interests.

\vskip 5mm


\section*{Acknowledgment}

The work is partially supported by  NSFC Grants (Nos. 12225103, 12071065 and 11871140)
and the National Key Research and Development Program of China (Nos. 2020YFA0713602 and 2020YFC1808301).


%

\end{document}